\documentclass[12pt, reqno]{amsart}
\usepackage[utf8]{inputenc}
\usepackage{amsmath, booktabs, amsthm, amssymb,amsfonts, diagbox, makecell, enumerate, xcolor, caption, bm, graphicx, makecell, mathtools, tikz, listings}
\usepackage[margin=1.25in]{geometry}
\makeatletter
\let\c@table\c@figure
\makeatother

\numberwithin{figure}{section}
\numberwithin{table}{section}

\usepackage{hyperref}
\hypersetup{
    colorlinks = true,
    linkcolor = blue,
    citecolor = blue,
    urlcolor = red
}

\graphicspath{}

\newcommand{\lapl}{\Delta}
\newcommand{\grad}{\nabla}

\newcommand{\norm}[1]{\lVert#1\rVert}

\renewcommand{\S}{\mathbb{S}}
\renewcommand{\H}{\mathbb{H}}

\newcommand{\R}{\mathbb{R}}

\newcommand{\B}{\mathbb{B}}
\newcommand{\D}{\mathbb{D}}
\newcommand{\csch}{\text{csch}}
\renewcommand{\coth}{\text{coth}}
\newcommand{\si}{\text{si}}
\newcommand{\ta}{\text{ta}}

\newtheorem{theorem}{Theorem}[section]
\newtheorem{lemma}[theorem]{Lemma}
\newtheorem{conjecture}[theorem]{Conjecture}

\title[Scaling inequalities and limits]{Scaling inequalities and limits for Robin and Dirichlet eigenvalues}
\author{Scott Harman}
\email{sharma65@illinois.edu}
\address{University of Illinois, Urbana, IL 61801, USA}


\begin{document}

\begin{abstract}
   For the Laplacian in spherical and hyperbolic spaces, Robin eigenvalues in two dimensions and Dirichlet eigenvalues in higher dimensions are shown to satisfy scaling inequalities analogous to the standard scale invariance of the Euclidean Laplacian. These results extend work of Langford and Laugesen to Robin problems and to Dirichlet problems in higher dimensions.
   In addition, scaled Robin eigenvalues behave exotically as the domain expands to a 2-sphere, tending to the spectrum of an exterior Robin problem.
\end{abstract}
\maketitle
\section{Introduction}
\subsection*{From Euclidean to spherical and hyperbolic}
In Euclidean space, we have the following scaling identity of eigenvalues for every $t \in \R$:
\[
t^2\lambda(t\Omega) = \lambda(\Omega)
\]
where $\lambda$ is a Dirichlet or Neumann Laplacian eigenvalue on a sufficiently regular domain $\Omega$. Robin eigenvalues $\lambda(\Omega, \alpha)$ with parameter $\alpha \in \R$ also satisfy a scale invariance provided one also scales the parameter, namely
\begin{equation}\label{robinEuclideanInvariant}
t^2 \lambda(t\Omega, \alpha/t) = \lambda(\Omega, \alpha).
\end{equation}
Equivalently, one can interpret these identities as the two functionals being constant for every value of $t$. However, when space is curved, the above identities fail, even after making sense of linear dilation on the new spaces. Our goal is to recover information about these functionals on spherical and hyperbolic spaces.

We will analyze the functionals
\[
\lambda_k(\Omega_t)S(t)^2
\]
where $\lambda_k$ is a Laplace--Beltrami eigenvalue, $\Omega_t$ is a spherical cap or geodesic ball, and $S(t)$ is some natural geometric scaling factor. For Robin eigenvalues, we will also look for natural normalizing factors $N(t)$ and their corresponding functionals
\[
\lambda_k(\Omega_t, \alpha / N(t))S(t)^2.
\]

The three main results in this paper are Robin monotonicity in spherical and hyperbolic space in two dimensions, Dirichlet monotonicity in spherical and hyperbolic space in higher dimensions, and computations of the limits for the Robin functionals. The limiting cases exhibit exotic behavior: as the spherical caps expand to fill up the entire sphere, the scaled Robin functional converges to a negative eigenvalue of the exterior Robin problem on the complement of a Euclidean disk.
\subsection*{Unifying spherical and hyperbolic notation}

Let $\S^n$ be the unit sphere of positive curvature $1$. We will employ the coordinate system $(\theta, \xi)$ where $\theta \in (0, \pi)$ is the angle of aperture from the north pole and $\xi \in \S^{n-1}$ represents the other angular components. Observe that any point $(\theta, \xi)$ has geodesic distance $\theta$ from the north pole. The domain of interest is
\[
C(\Theta) := \{(\theta, \xi) : 0 \leq \theta < \Theta, \xi \in \S^{n-1}\}
\]
where $0 < \Theta < \pi$, or geometrically, the spherical cap of aperture $\Theta$. We will consider the eigenvalue problem on $C(\Theta)$:
\[
-\lapl_{\S^n}u = \lambda u
\]
subject to Dirichlet, Neumann, or Robin boundary conditions.

We also consider hyperbolic space $\H^n$ of negative curvature $-1$. In this space, for $\Theta < 0$ we put $C(\Theta)$ as the geodesic disk of radius $|\Theta|$. There is the analogous coordinate system $(\theta, \xi)$ where $\theta \in (0, \infty)$ and $\xi \in \S^{n-1}$, and the corresponding eigenvalue problem
\[
-\lapl_{\H^n}u = \lambda u.
\]
Frequently we will denote the Laplace--Beltrami operator on these sets as
\[
\lapl_{C(\Theta)} := 
    \begin{cases}
     \lapl_{\H^n}, & -\infty < \Theta < 0, \\
     \lapl_{\S^n}, & 0 < \Theta < \pi.
    \end{cases}
\]

\subsection*{Two-dimensional Robin results}

Our goal is to quantify the effect that curving space has on the scaled eigenvalues, that is, how the geometry of the space interacts with its underlying analytic properties. In Euclidean $\R^n$, taking $\B^n$ as the unit ball, the constant function $t^2 \lambda(t \B^n)$ takes the appealing geometric form as an eigenvalue times the radius squared. In curved spaces such as $\S^n$ and $\H^n$, it is unclear what the ``scaling factor'' should be, i.e.,\ with what geometric quantity one should replace $t^2$.

In two dimensions, Langford and Laugesen in \cite[Theorem 1]{LL22} present results and scaling factors for Dirichlet and Neumann eigenvalues in two dimensions. A natural question to ask is whether analogous scaling results apply to Robin eigenvalues. We recall that the Robin problem is
\begin{gather*}
    \left\{
    \begin{aligned}
        & -\lapl_{C(\Theta)} u = \lambda u \quad \text{in } C(\Theta),\\
        & \frac{\partial u}{\partial n} + \alpha u = 0 \hspace{.9cm} \text{on } \partial C(\Theta),
    \end{aligned}
    \right.
\end{gather*}
where $\frac{\partial}{\partial n}$ is the outer unit normal derivative and $\alpha$ is a real-valued parameter. This problem induces a discrete spectrum of eigenvalues
\[
\lambda_1(\Theta, \alpha) < \lambda_2(\Theta, \alpha) \leq \lambda_3(\Theta, \alpha) \leq \dots \to \infty
\]
where $\lambda_k(\Theta, \alpha)$ is the $k^{\text{th}}$ Robin eigenvalue on $C(\Theta)$ with parameter $\alpha$.

The Euclidean scale invariance \eqref{robinEuclideanInvariant} suggests that the Robin parameter needs to be suitably scaled in the hyperbolic and spherical cases. Additionally, if the Robin parameter is negative, the spectrum now has negative eigenvalues. The Robin problem thus contrasts sharply with Neumann and Dirichlet problems, which have non-negative spectra. The negative eigenvalues of course tend to reverse the monotonicity of the functional, but the main interest for the negative eigenvalues is their peculiar limiting behavior. 

First, we handle positive Robin parameters (and hence positive Robin eigenvalues). A straightforward generalization of \cite[Theorem 1]{LL22} holds. The main adaptation for the Robin problem is to scale the Robin parameter $\alpha$ by the perimeter of the geodesic disk ($\sinh |\Theta|$) or by the perimeter of the spherical cap ($\sin \Theta$). At $\Theta = 0$, we define $\lambda_k(\D, \alpha)$ to be the $k^\text{th}$ Robin eigenvalue of the Euclidean unit disk $\D \subset \R^2$ with parameter $\alpha$.

\begin{theorem}[Positive eigenvalue monotonicity for Robin]\label{alphapos}
    Fix $n = 2$, $k \geq 1$, and $\alpha > 0$.

    \begin{enumerate}[(i)]
    \item The function \begin{equation}\label{tanhfunc}
        \Theta \mapsto 
        \begin{cases}
        \lambda_k(\Theta, \alpha/\sinh |\Theta|) \, 4 \tanh^2(\Theta/2), & \Theta \in (-\infty,0) , \\
        \lambda_k(\mathbb{D}, \alpha) , & \Theta=0 , \\
        \lambda_k(\Theta, \alpha/\sin \Theta) \, 4 \tan^2(\Theta/2) , & \Theta \in (0,\pi),
    \end{cases}
    \end{equation}
    increases strictly and continuously from $0$ to $\infty$. See Figure \ref{ta2fixedalpha}.
    \item The function \begin{equation}\label{sinhfunc}
        \Theta \mapsto 
        \begin{cases}
        \lambda_k(\Theta, \alpha/\sinh |\Theta|) \, \sinh^2(\Theta), & \Theta \in (-\infty,0) , \\
        \lambda_k(\mathbb{D}, \alpha), & \Theta=0 , \\
        \lambda_k(\Theta, \alpha/\sin \Theta) \, \sin^2(\Theta) , & \Theta \in (0,\pi),
        \end{cases}
    \end{equation}
    decreases strictly and continuously from $\infty$ to $0$. See Figure \ref{si2fixedalpha}.
    \end{enumerate}
\end{theorem}

The function $\eqref{tanhfunc}$ has the geometric interpretation of an eigenvalue times stereographic radius squared, where the stereographic transformation is detailed in \eqref{stereo}. The function $\eqref{sinhfunc}$ has the interpretation of an eigenvalue times perimeter squared.

Langford and Laugesen showed the Neumann case $\alpha = 0$ in \cite{LL22}. For the Dirichlet case (i.e.,\ $\alpha \to \infty$) the theorem was also proven in \cite{LL22}; in that case the function \eqref{tanhfunc} increases from $1$ to $\infty$.

For negative Robin parameters in the next theorem, we encounter more obstacles. The eigenvalue can turn negative, which reverses monotonicity and necessitates different techniques for the limiting values. In particular, a negative eigenvalue causes strange limiting behavior towards a seemingly unrelated unbounded Robin problem. 

We define $\lambda_k^\text{ext}(\D, \alpha)$
as the $k^\text{th}$ Robin eigenvalue on the exterior of the Euclidean unit disk $\D$. (The full definition of the PDE system is \eqref{exteriorpde}. More information on exterior Robin problems can also be found in \cite{KL18, KL20, KL23}.) We show that this exterior eigenvalue arises as a limit when the Robin eigenvalue is negative.

\begin{theorem}[Negative eigenvalue monotonicity for Robin]\label{alphaneg}
    Fix $n = 2$, $k \geq 1$, and $\alpha < 0$ where $\alpha \in [-m - 1, -m)$ and $m$ is a non-negative integer.

    \begin{enumerate}[(i)]
    \item If $k \leq 2m + 1$, then the eigenvalue  function \eqref{tanhfunc} is negative and decreases strictly and continuously from $0$ to $-\infty$ as a function of $\Theta \in (-\infty, \pi)$. If $k = 2m+2, 2m+3$, and $\alpha = -m-1$, then $\lambda_{2m+2} = \lambda_{2m+3} = 0$ for all $\Theta$ and \eqref{tanhfunc} is identically zero. Otherwise, $\lambda_k > 0$ for all $\Theta$ and Theorem \ref{alphapos}(i) continues to hold. See Figure \ref{ta2fixedalpha}.
    \\
    \item If $k \leq 2m + 1$, then the eigenvalue function \eqref{sinhfunc} is negative and increases strictly and continuously from $-\infty$ to $\lambda_k^\text{ext}(\D, \alpha)$ as a function of $\Theta \in (-\infty, \pi)$. If $k = 2m+2, 2m+3$, and $\alpha = -m-1$, then $\lambda_{2m+2} = \lambda_{2m+3} = 0$ for all $\Theta$ and \eqref{tanhfunc} is identically zero. Otherwise, $\lambda_k > 0$ for all $\Theta$ and Theorem \ref{alphapos}(ii) continues to hold. See Figure \ref{si2fixedalpha}.
    \end{enumerate}
\end{theorem}
The appearance of the exterior Robin eigenvalue problem is quite surprising, but we will provide some intuition when we prove the limiting values. Compared to the previous theorem, monotonicity and limits reverse sign when the eigenvalue is negative. Another fact that emerges from the theorem is that the normalized eigenvalues cannot change sign as a function of $\Theta$: they remain negative or positive (or zero) for all $\Theta$.

We will often make use of the following function definitions:
\[
\si(\Theta) = 
\begin{cases} 
\sinh(|\Theta|) & \text{for } \Theta < 0, \\
\sin(\Theta) & \text{for } \Theta \geq 0,
\end{cases} \hspace{0.3cm} \ta(\Theta) = 
\begin{cases} 
\tanh(|\Theta|) & \text{for } \Theta < 0, \\
\tan(\Theta) & \text{for } \Theta \geq 0.
\end{cases}
\]
These definitions differ from \cite{LL22} in using $\sinh (|\Theta|)$ for $\Theta < 0$ rather than $\sinh( \Theta)$.
 These functions make our statements more succinct, e.g.,\ the first part of Theorem \ref{alphapos} reads that $\lambda_k(\Theta, \alpha/\si(\Theta))4\ta^2(\Theta/2)$ is an increasing function. 
\begin{table}[t]
\centering
\parbox{.45\linewidth}{
    \centering
    
    \begin{tabular}{lcc}
        \toprule
        & $\Theta \to -\infty$ & $\Theta \to \pi$ \\
        \midrule
        $\lambda_k > 0$ & 0 & $\infty$ \\
        $\lambda_k < 0$ & 0 & $-\infty$ \\
        \bottomrule
    \end{tabular}
    \caption*{$\displaystyle \lim_\Theta \, \lambda_k(\Theta, \alpha/\si(\Theta))\,4\ta^2(\Theta/2)$}
}
\hfill
\parbox{.45\linewidth}{
    \centering
   
    \label{tab:table2}
    \begin{tabular}{lcc}
        \toprule
        & $\Theta \to -\infty$ & $\Theta \to \pi$ \\
        \midrule
        $\lambda_k > 0$ & $\infty$ & $0$ \\
        $\lambda_k < 0$ & $-\infty$ & $\lambda_k^\text{ext}(\D, \alpha)$ \\
        \bottomrule
    \end{tabular}
     \caption*{$\displaystyle \lim_\Theta \, \lambda_k(\Theta, \alpha/\text{si}(\Theta))\,\text{\si}^2(\Theta)$}
}
\caption{Summary of Robin eigenvalue limits}
\label{roblims}
\end{table}

\begin{figure}[htbp]
\includegraphics[width=\linewidth]{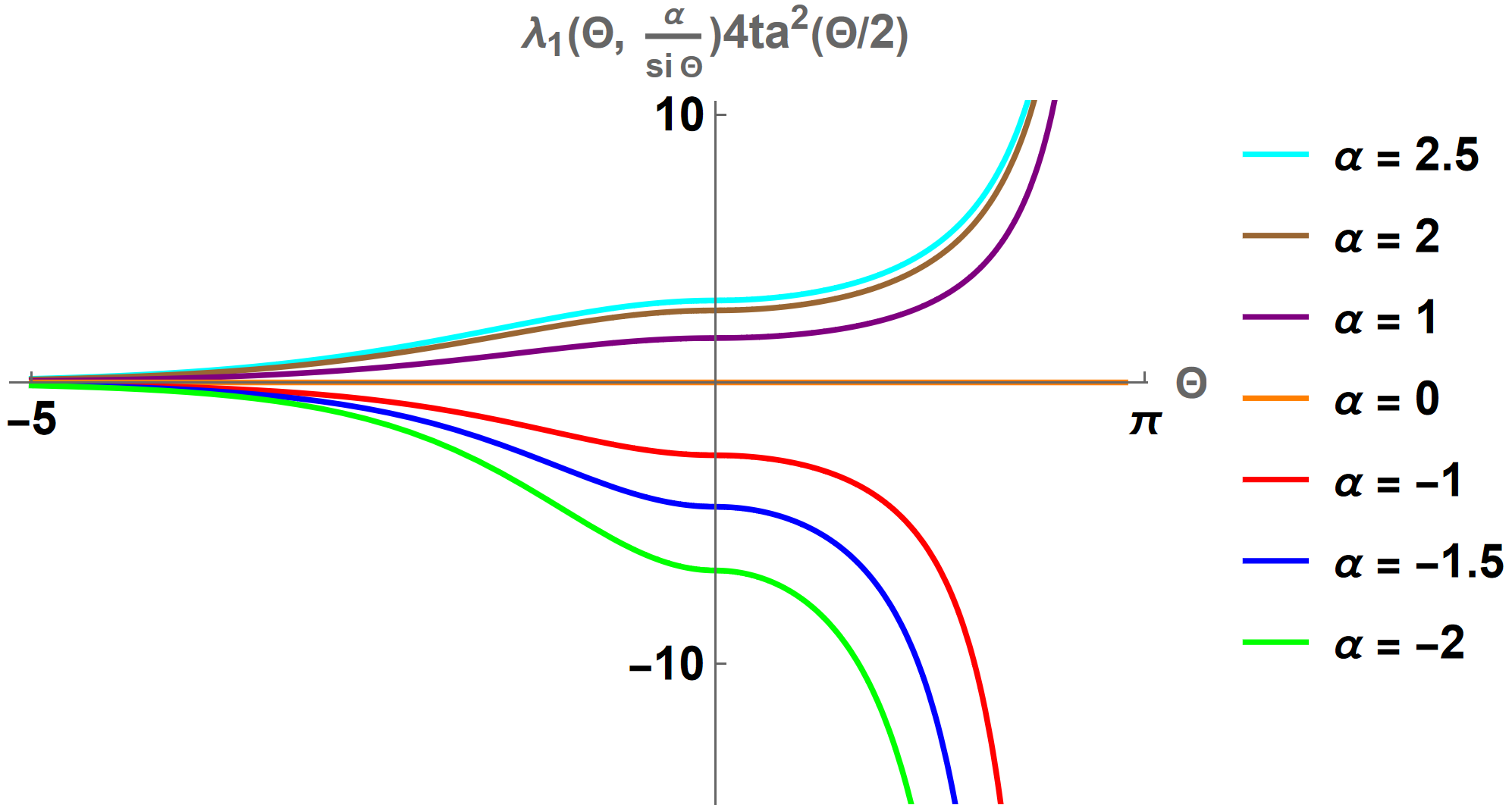}
\caption{Graph of the first Robin eigenvalue as a function of the aperture (or geodesic radius) $\Theta$ for fixed values of the Robin parameter $\alpha$. The parameter is normalized by perimeter and the eigenvalue scaled by stereographic radius. See Theorems \ref{alphapos}(i) and \ref{alphaneg}(i).}
\label{ta2fixedalpha}
\end{figure}
\begin{figure}[htbp]
\includegraphics[width=\linewidth]{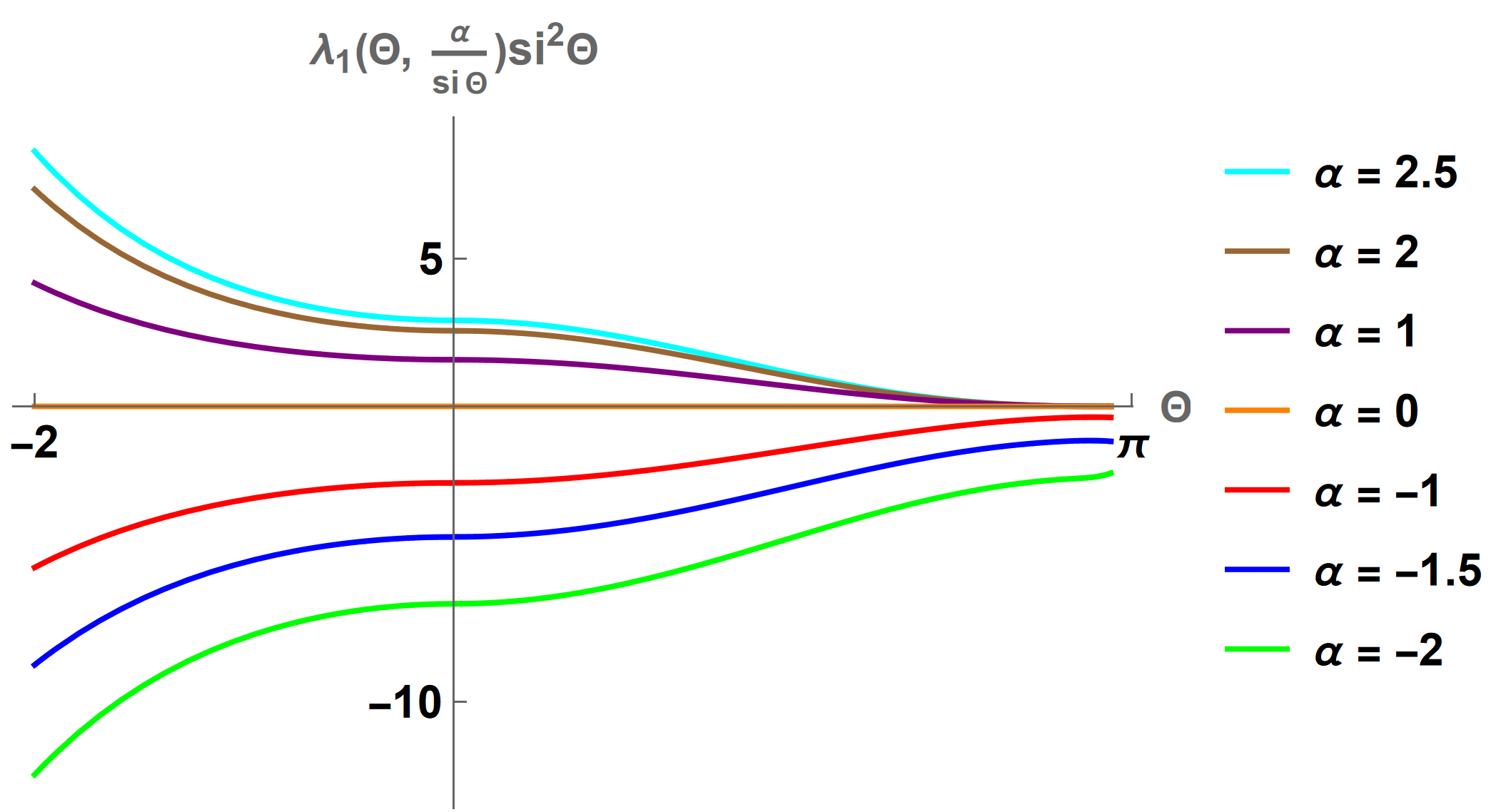}
\caption{Graph of the first Robin eigenvalue as a function of the aperture $\Theta$ for fixed values of the Robin parameter $\alpha$. The scaling factor is perimeter. Note the convergence to finite, nonzero limits for negative $\alpha$ values as $\Theta \to \pi$. See Theorems \ref{alphapos}(ii) and \ref{alphaneg}(ii).}
\label{si2fixedalpha}
\end{figure}

\begin{figure}[htbp]
\includegraphics[width=\linewidth]{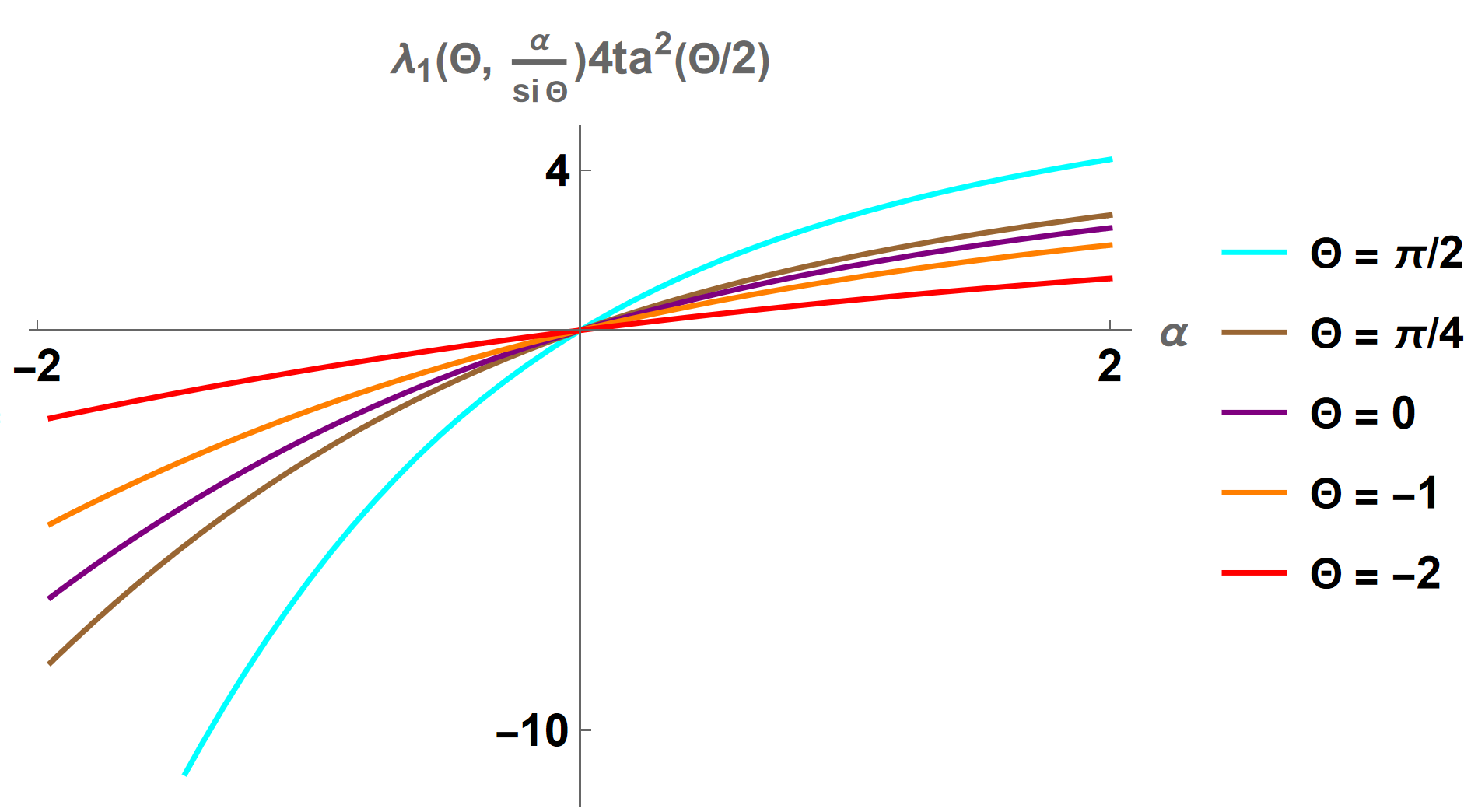}
\caption{Graph of the first Robin eigenvalue as a function of the Robin parameter $\alpha$ for fixed values of the aperture $\Theta$. The parameter has been normalized by perimeter and the eigenvalue scaled by stereographic radius. Note the reversal in ordering when crossing $\alpha = 0$.}
\label{ta2fixedtheta}
\end{figure}
\begin{figure}[htbp]
\includegraphics[width=\linewidth]{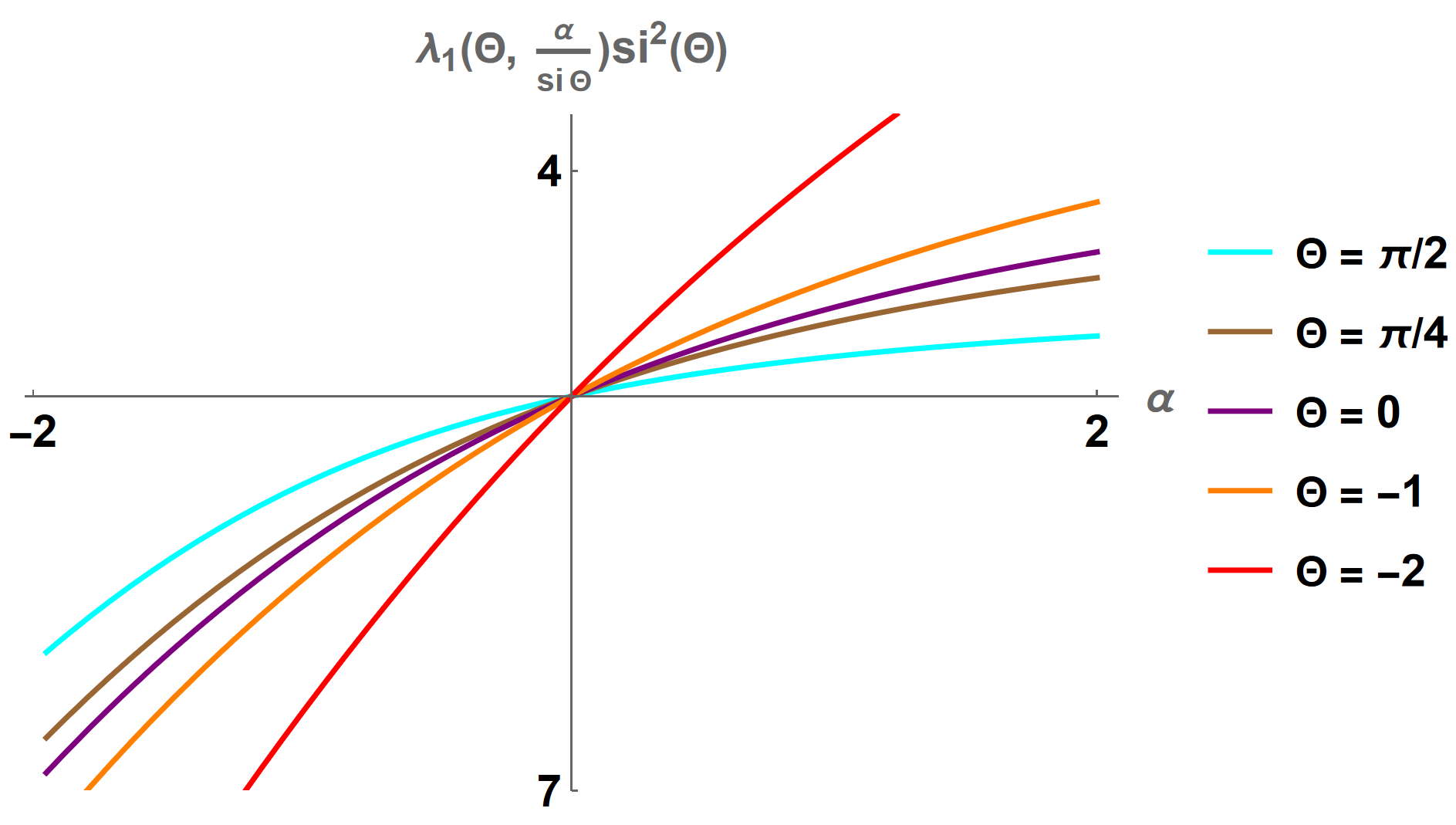}
\caption{Graph of the first Robin eigenvalue as a function of $\alpha$ for fixed values of $\Theta$. The eigenvalue is scaled and normalized by perimeter.}
\label{si2fixedtheta}
\end{figure}

Using this notation, we provide graphs of the first spectral curve ($k = 1$) in terms of $\alpha$ and $\Theta$. See Figures \ref{ta2fixedalpha}, \ref{si2fixedalpha}, \ref{ta2fixedtheta} and \ref{si2fixedtheta}. These plots help visualize Theorems \ref{alphapos} and \ref{alphaneg}. In addition, Table \ref{roblims} gives an overview of the limiting values.

\subsection*{Higher-dimensional results}

Higher dimensions pose completely different challenges to two dimensions. The Dirichlet integral is no longer conformally invariant and the proof techniques break down. In addition, previously natural choices of scaling factors in lower dimensions are no longer so in higher dimensions. The initial goal then is to determine which geometric quantities do suitably generalize to higher dimensions and scale the eigenvalue appropriately.

Table \ref{scaletab} gives examples of natural quantities on $\S^n$ which could potentially work, along with their geometric interpretation. The quantity $\sin(\Theta)$ is one such geometric factor that does not experience a monotonicity defect in higher dimensions, and it is this scaling quantity we focus on for $n \geq 3$. Volume also appears to be a monotonic scaling factor by the numerics in the Dirichlet case, but we postpone discussing this quantity for a brief moment.

Let $\lambda_k(\Theta)$ represent the $k^{\text{th}}$ Dirichlet eigenvalue of the Laplace--Beltrami operator on $C(\Theta)$. The scaling factor $\si^2(\Theta)$ is shown to produce monotonicity in higher dimensions, thus generalizing Langford and Laugesen's result to all dimensions, in the Dirichlet case. At $\Theta = 0$, we define $\lambda_k(\B^n)$ to be the $k^\text{th}$ Dirichlet eigenvalue of the Euclidean ball $\B^n \subset \R^n$.
\begin{theorem}[Higher-dimensional Dirichlet monotonicity]\label{main}
The function
\[
\Theta \mapsto 
\begin{cases}
\lambda_k(\Theta) \, \sinh^2(\Theta), & \Theta \in (-\infty,0) , \\
\lambda_k(\B^n) , & \Theta=0 , \\
\lambda_k(\Theta) \, \sin^2(\Theta) , & \Theta \in (0,\pi) ,
\end{cases}
\]
decreases strictly from $\infty$ to $0$, for all $k \geq 1$ and $n \geq 3$. See Figure \ref{dirichletlambda1graph}.
\end{theorem}

The result does not exclude other geometric scaling factors from working, nor does it state that other eigenvalues such as Neumann or Robin do not behave monotonically when scaled. Indeed, we will briefly discuss open problems below.
\subsection*{Open monotonicity problems}
The statement analogous to Theorem \ref{main} for the Neumann spectrum is supported by numerical evidence. However, the proof of the theorem does not translate well to the Neumann case, and so we state the analog as a conjecture. Let $\mu_k(\Theta)$  be the $k^{\text{th}}$ Neumann eigenvalue on $C(\Theta)$.
\begin{conjecture}[Higher-dimensional Neumann monotonicity] \label{neumannconj}
The function
\[
\Theta \mapsto 
\begin{cases}
\mu_k(\Theta) \, \sinh^2(\Theta), & \Theta \in (-\infty,0) , \\
\mu_k(\B^n) , & \Theta=0 , \\
\mu_k(\Theta) \, \sin^2(\Theta) , & \Theta \in (0,\pi/2) ,
\end{cases}
\]
decreases strictly from $\infty$ to $0$, for all $k \geq 2$ and $n \geq 3$. See Figure \ref{neumannlambdagraph}.
\end{conjecture}
The lower-dimensional case $n = 2$ of Conjecture \ref{neumannconj} was shown to hold by Langford and Laugesen in \cite[Theorem 1]{LL22}. One could further conjecture that the Robin case satisfies a monotonicity result in higher dimensions, but the proof would encounter similar obstacles to the Neumann case.

For another open problem, we look at the scaling factor volume. Write $n$-dimensional volume of the geodesic ball of radius $|\Theta|$ as
\[
V(\Theta) := c_n\int_0^{|\Theta|} \si(\theta)^{n-1} \, d\theta,
\]where $c_n$ is the appropriate normalizing factor, equal to the surface area of the Euclidean $(n-1)$-sphere $\S^{n-1}$. Consider the scaled Dirichlet and Neumann eigenvalues \[\lambda_k(\Theta)V(\Theta)^{2/n}, \quad \mu_k(\Theta)V(\Theta)^{2/n},
\]with $\lambda_k(\Theta)V(\Theta)^{2/n} = \lambda_k(\B^n)|\B^n|^{2/n}$ at $\Theta = 0 $ by convention. This scaling factor is somewhat peculiar since Figure \ref{surfacemonofigs} numerically supports that $\lambda_k(\Theta)V(\Theta)^{2/n}$ is always decreasing whereas Figure \ref{neummanincreasingfig} numerically supports that $\mu_k(\Theta)V(\Theta)^{2/n}$ is always increasing.

We look at the Dirichlet case first, for $k =1,2$. Observe that after a change of variable $t = \theta/\Theta$, we have for $\Theta \in (0, \pi)$ that
\[
\frac{V(\Theta)}{\Theta^n} = c_n \int_0^1 \left(\frac{\sin(\Theta t)}{
\Theta t}\right)^{\! \! n-1} t^{n-1} dt,
\]implying that $V(\Theta)^{2/n}/\Theta^2$ is a decreasing function since $\sin(x)/x$ is decreasing.  (The $\Theta < 0$ case is proven similarly.) One may also verify that the displayed ratio converges to $|\B^n|$ as $\Theta \to 0$. Since $\lambda_1(\Theta)\Theta^2$ decreases from $\infty$ to $0$ by \cite[Theorem 3]{LL22}, we see in all dimensions $n \geq 2$ that
\[\lambda_1(\Theta)V(\Theta)^{2/n}\]
decreases strictly from $\infty$ to $0$ for $\Theta \in (-\infty, \pi)$, passing through the value $\lambda_k(\B^n)|\B^n|^{2/n}$ continuously at $\Theta = 0$. Note that \cite[Theorem 3]{LL22} also manifests as a consequence of a theorem by Cheng (see \cite{C75} and \cite[Chapter III]{C84}.) Furthermore, for the second eigenvalue, from \cite[Theorem 4]{LL22} we see for all $n \geq 2$ that
\[
\lambda_2(\Theta)V(\Theta)^{2/n}
\]
decreases strictly from $\infty$ to $\lambda_k(\B^n)|\B^n|^{2/n}$  for $\Theta \in (-\infty, 0]$. However, we cannot make further deductions on the behavior of the function on the interval $(0, \pi)$ using this theorem. 

For the Neumann case, \cite[Theorem 6]{LL22} implies in dimension $n =2$
\[
\mu_2(\Theta)V(\Theta)^{2/n}
\]
increases strictly and continuously on $(-\infty, \pi)$ from $2\pi$ to $8\pi$. Observe how the function converges to non-zero values, highlighting the peculiar behavior of this scaling factor.

We state monotonicity conjectures for these two functionals. 

\begin{conjecture}[Dirichlet monotonicity with volume scaling]\label{surfaceareaconjdirch} The function
\[
\lambda_k(\Theta) \, V(\Theta)^{2/n}
\]
decreases strictly for $\Theta \in (-\infty, \pi)$ for $k \geq 2$ and $n \geq 2$. See Figure \ref{surfacemonofigs}.
\end{conjecture}

\begin{conjecture}[Neumann monotonicity with volume scaling]\label{surfaceareaconjneum} The function
\[
\mu_k(\Theta) \, V(\Theta)^{2/n}
\]
increases strictly for $\Theta \in (-\infty, \pi)$ for $k \geq 2$ and $n \geq 2$. See Figure \ref{neummanincreasingfig}.
\end{conjecture}
(The Neumann case where both $k = 2$ and $n =2$ has already been proven; see Langford and Laugesen \cite[Theorem 6]{LL22}.) 

The limiting values of the functionals are not obvious. Certainly at the origin they should pass through $\lambda_k(\B^n)|\B^n|^{2/n}$ and $\mu_k(\B^n)|\B^n|^{2/n}$ but the endpoint $\Theta = -\infty$ is not clear. The Neumann case $k=2, n =2$ has already been shown to converge to a non-zero limit as $\Theta \to -\infty$, and Figure \ref{neummanincreasingfig} supports that this phenomenon continues in higher dimensions. For the Dirichlet eigenvalues, since $\lambda_1(\Theta)V(\Theta)^{2/n} \to \infty$ as $\Theta \to -\infty$, it must also be the case that $\lambda_k(\Theta)V(\Theta)^{2/n} \to \infty$. For $\Theta \to \pi$, the Dirichlet and Neumann spectra should converge to the spectrum of the whole sphere; see \cite{BKN10} and \cite{BKN19}. More investigation and analysis is needed for this scaling factor. 

\renewcommand{\arraystretch}{1.5}
\begin{table}[t]
\begin{center}
\begin{tabular}{ c  l  l  } 
 \toprule
 Scaling Factor & Geometric Interpretation & Result\\ 
\midrule
 $\sin (\Theta)$ & Perimeter of boundary circle & Monotonic by Theorem \ref{main}  \\   
  
 $\Theta$ & Geodesic radius  & Non-monotonic by Figure \ref{geononmono}\\ 
 
 $\tan(\Theta/2)$ & Stereographic radius & Non-monotonic by Figure \ref{stereononmono} \\ 
 
\rule{0pt}{2.5em}
$V(\Theta)$ & Volume & \shortstack[l]{Open problem \\ ---See Conjectures \ref{surfaceareaconjdirch} and \ref{surfaceareaconjneum}}
  \\[2ex]
 \bottomrule
\end{tabular}
\caption{Dirichlet scaling factor results and conjectures in dimensions $n \geq 3$.} \label{scaletab}
\end{center}
\end{table}

\subsection*{Related results}

Upper and lower bounds and some exact values are known for the Dirichlet eigenvalues in both hyperbolic and spherical spaces. For spherical caps in dimension $n = 3$, we have an exact formula for the first Dirichlet eigenvalue $\lambda_1(\Theta) = \pi^2/\Theta^2 - 1$, which one may verify is monotonically decreasing after multiplying by the scaling factor $\sin^2(\Theta)$. From Borisov and Freitas \cite[Theorem 3.3]{BF17}, we provide a two-sided estimate on the first eigenvalue, which we compute here for dimension $n = 4$ as
\begin{align*}
\frac{\lambda_1(\B^4)}{\Theta^2} - 2 &\leq \lambda_1(\Theta) \leq \frac{\lambda_1(\B^4)}{\Theta^2} - \frac{3}{4}\left(\frac{1}{\sin^2(\Theta)} - \frac{1}{\Theta^2} \right) \hspace{.41cm} \text{for } \Theta > 0, \\
\frac{\lambda_1(\B^4)}{\Theta^2} + 2 &\leq \lambda_1(\Theta) \leq \frac{\lambda_1(\B^4)}{\Theta^2} + \frac{3}{4}\left(\frac{1}{\sinh^2(\Theta)} - \frac{1}{\Theta^2} \right) \hspace{.2cm} \text{for } \Theta < 0. 
\end{align*}
Such bounds cannot prove monotonicity. However, they are sharper estimates than the bounds acquired from Theorem \ref{main}, which are
\begin{align*}
\lambda_1(\Theta) &\leq \frac{\lambda_1(\B^4)}{\sin^2(\Theta)}\hspace{.58cm} \text{for } \Theta > 0, \\
\lambda_1(\Theta) &\geq \frac{\lambda_1(\B^4)}{\sinh^2(\Theta)} \hspace{.36cm} \text{for } \Theta < 0. 
\end{align*}
There is other literature involving ratios of eigenvalues and Faber--Krahn type inequalities by Ashbaugh, Benguria, and Linde \cite{AB95, AB01, BL07}. Estimates for higher Dirichlet eigenvalues can be found in a paper by Berge \cite{B22}.

\section{Two-dimensional Robin monotonicity --- Theorems 1.1 and 1.2}
For this section, we will work exclusively in two dimensions and prove the monotonicity statements in Theorems \ref{alphapos} and \ref{alphaneg}. The proofs of the limiting values are in the next section.

Let $C(\Theta)$ remain as defined above, either as a hyperbolic or spherical cap. 
We consider the Robin eigenvalue problem
\begin{gather*}
    \left\{
    \begin{aligned}
        & -\lapl_{C(\Theta)} u = \lambda u \quad \text{in } C(\Theta), \\
        & \frac{\partial u}{\partial n} + \alpha u = 0 \hspace{0.9cm} \text{on } \partial C(\Theta),\\
    \end{aligned}
    \right.
\end{gather*}
with discrete spectrum
\[
\lambda_1(\Theta, \alpha) < \lambda_2(\Theta, \alpha) \leq \lambda_3(\Theta, \alpha) \leq \dots \to \infty.
\]
The variational characterization of $\lambda_k(\Theta, \alpha)$ is
\[
\lambda_k(\Theta, \alpha) = \min_{\mathcal{L}} \max_{0 \neq u \in \mathcal{L}} \frac{\displaystyle\int_{C(\Theta)} |\grad_{\S^2} u|^2\, dA + \alpha \int_{\partial C(\Theta)} |u|^2\, dS}{\displaystyle\int_{C(\Theta)} |u|^2\, dA}
\]
where $\mathcal{L}$ varies over $k$-dimensional subspaces of $H^1(C(\Theta))$, $dA$ is the 2-dimensional area element, and $dS$ is the appropriate 1-dimensional measure on the boundary. When $\alpha < 0$, at least one negative eigenvalue exists, which impacts monotonicity calculations and proofs. We will separate then our monotonicity discussion into two cases: positive $\alpha$ parameters and negative $\alpha$ parameters.

\subsection*{Positive Robin parameter --- Theorem 1.1}
\begin{proof}
 Fix $\alpha > 0$. We perform the following conformal transformation to the plane under polar coordinates:

\begin{equation}\label{stereo}(\theta, \xi) \mapsto \begin{cases}
 (\tan \frac{\theta}{2}, \xi) & \text{if } \Theta > 0, \\
 (\tanh \frac{\theta}{2}, \xi) & \text{if } \Theta < 0.
\end{cases}
\end{equation}
This map is exactly stereographic projection in the spherical cap case, mapping the spherical cap $C(\Theta)$ to a disk in the plane.

The conformal map defined above transforms the Rayleigh quotient and eigenequation in 
a systematic way. Define the weight functions
\[
w_\pm(r) = \frac{4}{(1 \pm r^2)^2}.
\]
Consider the eigenequation $-\lapl_{C(\Theta)} u = \lambda_k(\Theta, \alpha) u$. Letting $v$ be the pushforward of $u$ under the above map, the equation transforms as
\begin{equation*}\label{weightedlapl}
-\lapl_{\R^2} v = \lambda_k(\Theta, \alpha) w_{\pm} v
\end{equation*}
where the weight is $w_+$ for positively curved spherical caps and $w_-$ for negatively curved geodesic disks. One can also consult \cite[Section 5]{LL22} for more detail on the transformations.

We define $R := \tan(\Theta/2)$ (or $\tanh(|\Theta|/2)$ in the hyperbolic case) for the radius of the disk where $v$ is defined. The Rayleigh quotient for $\lambda_k(\Theta, \alpha)$ transforms by conformal invariance of the Dirichlet integral to

\[
\frac{ \displaystyle \int_{R \D} |\grad v(r, \phi)|^2 \, dA + \alpha \sqrt{w_\pm(R)} \int_{\partial (R \D)} v(R, \phi)^2 \, d\phi}{\displaystyle \int_{R\D} v(r, \phi)^2 w_{\pm}\, dA}
\]
where $dA = r \, dr d\phi$.
The rescaling $g(z) = v(Rz)$ further transforms the expression into
\[
\frac{ \displaystyle \int_{\D} |\grad g(r, \phi)|^2\, dA + \alpha \sqrt{w_\pm(R)} \int_{\partial \D} g(1, \phi)^2 \, d\phi}{\displaystyle \int_{\D} g(r, \phi)^2 w_\pm (Rr) R^2\,dA}.
\]
After normalizing $\alpha$ by the factor present in Theorem \ref{alphapos}, and noting that $\sin \Theta = \sqrt{w_+(R)}$ and $\sinh |\Theta| = \sqrt{w_-(R)}$, the Rayleigh quotient for $\lambda_k(\Theta, \alpha/\si (\Theta))$ becomes
\[
\frac{ \displaystyle \int_{\D} |\grad g(r, \phi)|^2\, dA + \alpha \int_0^{2\pi} g(1, \phi)^2\, d \phi}{\displaystyle \int_{\D} g(r, \phi)^2 w_\pm (Rr) R^2\, dA}.
\]
As the numerator is independent of $R$ and positive (recall $\alpha > 0$), to show parts (i) and (ii) of the theorem we need only show monotonicity with respect to $R$ in the denominator after multiplying by the appropriate scaling factor. The scaling factors transform as
\[
\si^2(\Theta) = \frac{4R^2}{(1\pm R^2)^2}, \qquad 4\ta^2(\Theta/2) = 4R^2,
\]
where $\si (\Theta)$ and $\ta(\Theta)$ were defined in the remark after Theorem \ref{alphaneg}.
Let $S(R)$ be equal to the appropriate scaling factor.
The variational characterization for the scaled eigenvalue is thus
\begin{equation}\label{finalRobinRQ}
\lambda_k(\Theta, \alpha/\si(\Theta)) S(R):= \min_{\mathcal{L}} \max_{0 \neq g \in \mathcal{L}}
\frac{ \displaystyle \int_{\D} |\grad g(r, \phi)|^2 \, dA + \alpha \int_0^{2\pi} g(1, \phi)^2 \, d \phi}{\displaystyle \int_{\D} g(r, \phi)^2 \frac{w_\pm (Rr)}{S(R)} R^2 \, dA}
\end{equation}
where $\mathcal{L}$ varies over $k$-dimensional subspaces of $H^1(\D)$. The crucial point is that the only dependence on $R$ in the expression is the quantity
\[
\frac{w_{\pm}(Rr)}{S(R)}R^2
\]
in the denominator. Langford and Laugesen's proof of \cite[Proposition 9]{LL22} analyzes this quantity and justifies the strict monotonicity statements in Theorem \ref{alphapos}.
\end{proof}
We still need to show the limiting values in the theorem are correct, but we will delay the proof until after analyzing the negative parameters.

\subsection*{Negative Robin parameter --- Theorem 1.2}

Negative parameters require extra care as the eigenvalue could be negative, which reverses the monotonicity results. For spherical caps and geodesic disks, the behavior of negative eigenvalues is fortunately quite tractable, as seen in the following lemma.

\begin{lemma}[Counting negative eigenvalues]\label{unifpos}
    Let $\alpha < 0$ and $\Theta \neq 0$. If $\alpha \in [-m-1, -m)$, then the first $2m+1$ perimeter-normalized eigenvalues $\lambda_k(\Theta, \alpha/\emph{\si} (\Theta))$ are negative. The other eigenvalues are positive, except that when $\alpha = -m-1$ there are two zero eigenvalues, $\lambda_{2m+2} = \lambda_{2m+3} = 0$. The analogous results hold when $\Theta = 0$ for $\lambda_k(\D, \alpha)$.
\end{lemma}

\begin{proof}
The proof relies on the fundamental connection between Robin and Steklov eigenvalues, and so we will mostly focus on integer $\alpha$. Consider the Steklov problem on the unit disk $\D$, which is
\begin{gather*}
    \left\{
    \begin{aligned}
        & -\lapl_{\R^2} g = 0  \hspace{.35cm} \text{in } \D,\\
        & \frac{\partial g}{\partial n} = \sigma g \hspace{0.97cm} \text{on } \partial \D.
    \end{aligned}
    \right.
\end{gather*}
The eigenfunctions and eigenvalues of this problem are well-understood (see \cite{BFK17} and \cite{GP17}.) The eigenvalues repeated according to multiplicity are
\[\sigma \in \{0, 1, 1, 2, 2, 3, 3, \dots \}.\]
Suppose $\Theta > 0$. We will transform the Steklov problem on the disk into a Steklov problem on a spherical cap by reversing the transformations performed in the preceding section. Letting $R := \ta(\Theta/2)$, we perform the rescaling $v(z) = g(z/R)$, so that $v$ is defined on the disk $R \D$. Then, we take the pullback of $v$ along the conformal map $\eqref{stereo}$ and call the pullback $u(\theta, \xi)$, whose domain is the spherical cap $C(\Theta)$. Under these conformal mappings, the Steklov problem on the disk transforms into the perimeter-normalized Steklov problem
\begin{gather*}
    \left\{
    \begin{aligned}
        & -\lapl_{\S^2} u = 0 \hspace{2.2cm} \text{in } C(\Theta),\\
        & \frac{\partial u}{\partial n} + (\alpha / \si (\Theta)) u = 0 \hspace{.6cm} \text{on }\partial C(\Theta),
    \end{aligned}
    \right.
\end{gather*}
where $\alpha = -\sigma$. 

So, the perimeter-normalized Robin problem on a spherical cap or geodesic disk has a zero eigenvalue, $\lambda(\Theta, \alpha/\si (\Theta)) = 0$, precisely when $\alpha$ is a non-positive integer. Recall that Robin eigenvalues are strictly increasing and continuous as a function of $\alpha$. This fact along with the systematic form of the Steklov spectrum on the disk implies there are at least $2m+1$ negative eigenvalues when $\alpha = -m-1$. If there were more than $2m+1$, then the Neumann spectrum $(\alpha = 0)$ would have a negative eigenvalue, which is impossible. It is easy to see $\lambda_{2m+1} < \lambda_{2m+2} = \lambda_{2m+3} = 0$ when $\alpha = -m -1$, as one simply uses the Steklov eigenfunctions as Robin ones. Furthermore, $\lambda_{2m+1} < 0 < \lambda_{2m+2}$ when $\alpha \in (-m-1, m)$.

For $\Theta = 0$, the argument is even simpler as one does not need to perform any conformal mappings and the result is immediate.
\end{proof}
The graphs in \cite{BFK17} give a detailed picture of what is occurring with the Robin eigencurves on the disk. The perimeter normalization is crucial for us so that the number of negative eigenvalues is independent of $\Theta$. That is, if $\lambda_k(\Theta, \alpha/\si(\Theta)) < 0$ for some $\Theta$, then it remains negative for all $\Theta$ values.

We can now show the monotonicity statements in Theorem \ref{alphaneg}.
All calculations resulting in \eqref{finalRobinRQ} are still valid, and Theorem \ref{alphapos} continues to hold for positive eigenvalues. For negative eigenvalues, one notes that if we already know that $\lambda_k(\Theta, \alpha/\si(\Theta)) < 0$, then we can modify the family of subspaces $\mathcal{L}$ in \eqref{finalRobinRQ} to only be those for which the maximizer of the Rayleigh quotient is negative. Using the lemma above, the negativity of the eigenvalue is independent of $\Theta$, and so we can use this modified variational characterization for each value of $\Theta$.

Now, since the numerator is uniformly negative in \eqref{finalRobinRQ}, the monotonicity results reverse direction, and we arrive at the desired statements in Theorem \ref{alphaneg}.
\section{Positive Robin limiting values for Theorem \ref{alphapos}}

We move forward to proving the limits for the scaled Robin eigenvalues. These proofs are long and involved, so we split them over three sections. A plethora of disparate techniques will be employed, with the negative eigenvalues in the next two sections requiring the most interesting and challenging ones. In this section, we prove the limits for positive eigenvalues in Theorems \ref{alphapos}. Each functional has three limits that we need to compute: $\Theta \to -\infty, \Theta \to 0,$ and $\Theta \to \pi$. 

The first and simplest case is when $\Theta \to 0$. We want to show it converges to the Euclidean eigenvalue. Observe that in \eqref{finalRobinRQ}, simple computations show the quantity $w_\pm(Rr)R^2/S(R) \to 1$ uniformly as $R \to 0$  and we recover the Euclidean Rayleigh quotient and eigenvalues.

There are quite a few cases to consider for the two other limiting values of $\Theta$, so we will be doing casework. Before doing so, we state the following Lemma, a Robin generalization of \cite[Lemma 8]{LL22} which states that certain weighted Euclidean eigenvalues have an infinite limit. The Laplacian $\lapl$ in the following lemma is interpreted as the Euclidean Laplacian $\lapl_{\R^2}$.

\begin{lemma}(Pointwise vanishing of the weight)\label{weightedDivergesLemma}
    Let $\Omega \subset \R^2$ be a bounded, regular domain. Fix $\alpha \in \R$. For each $R > 0$, let $w_R: \Omega \to (0,1]$ be a measurable function with $\text{ess inf }w_R > 0.$ Let $\lambda_k(w_R, \alpha)$ be the $k^\text{th}$ Robin eigenvalue of $-w_R^{-1}\lapl.$ Then, there exists an index $N$ independent of $R$ such that $\lambda_k(w_R, \alpha) > 0$ if and only if $k \geq N$. Furthermore, if $\lim_{R \to \infty} w_R = 0$ a.e., then 
    \[
    \lim_{R \to \infty}\lambda_k(w_R, \alpha) = \infty \quad \text{for each $k \geq N$.}
    \]
\end{lemma}
\begin{proof}

First, we show the existence of the index $N$. The variational characterization reads
\begin{equation}\label{weightedRQ}
\lambda_k(m_R, \alpha) = \min_{\mathcal{L}} \max_{0 \neq f \in \mathcal{L}} \frac{\displaystyle \int_\Omega |\grad f|^2\, dA + \alpha \int_{\partial \Omega}|f|^2\, dS}{\displaystyle \int_\Omega |f|^2 w_R\, dA}
\end{equation}
where $\mathcal{L}$ varies over $k$-dimensional subspaces of $H^1(\Omega)$, $dA = r \, dr d\phi$ and $dS$ is the appropriate 1-dimensional Hausdorff measure. 

For each $R$, let $N_R$ be the first index $k$ such that $\lambda_{N_R}(w_R,\alpha) > 0$. Let $R_1, R_2$  be arbitrary positive $R$-values. If $N_{R_1} = 1$, then the spectrum is positive, and so $\alpha > 0$ and $N_{R_1} = N_{R_2} = 1$. So, assume $N_{R_1} \geq 2$ and $k < N_{R_1}$. Then, $\lambda_k(w_{R_1}, \alpha) \leq 0$.  By homogeneity of the Rayleigh quotient, we may assume that each trial function $f$ in \eqref{weightedRQ} is $L^2$-normalized so that $\int_\Omega |f|^2 \, dA = 1$. Then, the denominator of the Rayleigh quotient is bounded above and below:
\[
\text{ess inf }w_{R_1} \leq \int |f|^2 w_{R_1} \, dA \leq 1.
\]Since the denominator is bounded and $\lambda_k(w_{R_1}, \alpha) \leq 0$, the variational characterization \eqref{weightedRQ} implies there exists a $k$-dimensional subspace $\mathcal{M}$ of $H^1(\Omega)$ such that
\[
\max_{0 \neq f \in \mathcal{M}} \left(\displaystyle \int_\Omega |\grad f|^2\, dA + \alpha \int_{\partial \Omega}|f|^2\, dS\right) \leq 0.
\]
Using $\mathcal{M}$ as a trial subspace in the variational characterization of $\lambda_k(w_{R_2}, \alpha)$ implies $\lambda_k(w_{R_2}, \alpha) \leq 0$. It follows that $N_{R_1} \leq N_{R_2}$, and $N_{R_1}=N_{R_2}$ by symmetry. Thus, the index $N_R$ is indeed independent of $R$ and we may write it simply as $N$.

We move on to the second statement in the lemma. Fix $0 < \delta < 1$ and define a new weight $m_R := \max (\delta, w_R)$ so that $m_R$ is uniformly bounded away from zero. The $k^{th}$ Robin eigenvalue of the operator $-m_R^{-1}\lapl$ has variational characterization
\begin{equation}\label{deltaRQ}
\lambda_k(m_R, \alpha) = \min_{\mathcal{L}} \max_{0 \neq f \in \mathcal{L}} \frac{\displaystyle \int_\Omega |\grad f|^2\, dA + \alpha \int_{\partial \Omega}|f|^2\, dS}{\displaystyle \int_\Omega |f|^2 m_R\, dA}.
\end{equation}
Assume $k \geq N$. Then, the numerator of the maximizer is positive, and by comparing Rayleigh quotients, we see that $\lambda_k(w_R, \alpha) \geq \lambda_k(m_R, \alpha)$ for each $R$. Furthermore, since $m_R$ is a function into $(0,1]$ with $\text{ess inf }m_R > 0$, we see that $\lambda_k(m_R, \alpha) > 0$ as well.

By the variational characterization, we see that
\begin{equation}\label{deltaBounds}
\delta^{-1} \lambda_k(1, \alpha) \geq \lambda_k(m_R, \alpha) \geq \lambda_k(1, \alpha)
\end{equation}
where $\lambda_k(1, \alpha)$ is the $k^\text{th}$ Robin eigenvalue of the unweighted Laplacian. We will show that $\liminf_{R \to \infty} \lambda_k(m_R, \alpha) \geq \delta^{-1}\lambda_N(1, \alpha)$ which implies $\lambda_k(w_R, \alpha) \to \infty$ as $R \to \infty$.

For each $R$, take a $k^\text{th}$ Robin eigenfunction $f_R$ of $-m_R^{-1}\lapl$. The eigenfunction satisfies the weak equation
\begin{equation}\label{weakrobineq}
\int_\Omega \grad \phi \cdot \grad f_R \, dA + \alpha \int_{\partial \Omega} \phi f_R \, dS= \lambda_k(m_R, \alpha) \int_\Omega \phi f_R m_R \, dA
\end{equation}
for all $\phi \in H^1(\Omega)$. We normalize the eigenfunction $f_R$ in $L^2(\Omega)$ so that $\int_\Omega |f_R|^2 = 1$. Now take a sequence of $R$ values tending to $\infty$. We see that
\[
\int_\Omega  |\grad f_R|^ 2\, dA + \alpha \int_{\partial \Omega}  f_R^2 \, dS = \lambda_k(m_R, \alpha) \int_\Omega f_R^2 m_R \, dA \leq \delta^{-1}\lambda_k(1, \alpha) \int_\Omega f_R^2 \, dA
\]
where the inequality used \eqref{deltaBounds} and that $m_R \leq 1$.
Applying the $L^2$ normalization of $f_R$, we derive an upper bound $\delta^{-1}\lambda_k(1, \alpha)$ which is independent of $R$.

When $\alpha \geq 0$, the above estimate gives us a bound on the $H^1$ norm of $f_R$ that is independent of $R.$ For $\alpha < 0$, some more work has to be done. From the calculations done in \cite[Chapter 5]{L12}, justified by the proof of the trace theorem in Evans \cite[Section 5.5]{E10}, one deduces that
\[
\int_\Omega |\grad f_R|^2 \, dA + \alpha \int_{\partial \Omega} |f_R|^2 \, dA \geq \frac{1}{2} \int_\Omega |\grad f_R|^2 \, dA + C\alpha \int_{\Omega} |f_R|^2 \, dA
\]
for some constant $C$ depending only on $\Omega$. Since $f_R$ is normalized in $L^2$, we again derive a uniform upper bound for the $H^1$ norm of $f_R$ that is independent of $R$.

After applying the Rellich--Kondrachov theorem, we pass to a subsequence of $f_R$ and yield a function $f \in H^1(\Omega)$ such that $f_R \rightharpoonup f$ weakly in $H^1(\Omega)$ and $f_R \to f$ strongly in $L^2(\Omega)$. Furthermore, by \cite[Corollary 18.4]{L17}, the trace operator $T: H^1(\Omega) \to L^2(\partial \Omega)$ is compact, and since $f_R \rightharpoonup f$ weakly in $H^1$, $Tf_R \to T_f$ strongly in $L^2(\partial \Omega).$

We may further assume $\lambda_k(m_R, \alpha)$ converges to a finite limit as $R \to \infty$, after extracting a subsequence. As the weak Robin eigenequation \eqref{weakrobineq} is preserved under weak limits, the calculations in \cite[Lemma 8]{LL22} still apply, and we deduce
\[
- \lapl f = (\delta \lim_{R \to \infty} \lambda_k(m_R, \alpha)) f
\]
in the weak sense. (This step uses the fact that $w_R \to 0$ a.e.\ to imply $m_R \to \delta$ a.e.)
Therefore, $f$ is an eigenfunction of the unweighted Robin Laplacian on $\Omega.$ The eigenvalue $\delta \lim_{R \to \infty} \lambda_k(m_R, \alpha)$ must be positive by \eqref{deltaBounds} since $k \geq N$. It follows that $\delta \liminf_{R \to \infty} \lambda_k(m_R, \alpha) \geq \lambda_N(1, \alpha)$, concluding the proof.
\end{proof}

\subsection*{Scaling factor $4\ta^2(\Theta/2)$}

Now we prove the limits for Theorem \ref{alphapos}(i). Suppose either $\alpha > 0$ and $k \geq 1$ or $\alpha \in [-m-1, -m)$ and $k \geq 2m+2$ ($\geq 2m+4$ if $\alpha = -m-1$.) By Lemma \ref{unifpos}, we know $\lambda_k(\Theta, \alpha / \si(\Theta)) > 0$. 

First consider $\Theta \to \pi$. We apply the lemma above to the weighted eigenvalue $\lambda_k(\Theta, \alpha/\sin \Theta)4\tan^2(\Theta/2)$. The variational characterization \eqref{finalRobinRQ} reads (recalling that $4\tan^2(\Theta/2) = 4R^2$)
\[
\lambda_k(\Theta, \alpha/\sin \Theta ) 4\tan^2(\Theta/2):= \min_{\mathcal{L}} \max_{0 \neq g \in \mathcal{L}}
\frac{ \displaystyle \int_{\D} |\grad g(r, \phi)|^2\, dA + \alpha \int_0^{2\pi} g(1, \phi)^2\, d \phi}{\displaystyle \int_{\D} g(r, \phi)^2 (1+R^2r^2)^{-2}\,dA}.
\]
The weights $w_R := (1+R^2r^2)^{-2}$ satisfy the hypotheses of Lemma \ref{weightedDivergesLemma}, and it follows that $\lambda_k(\Theta, \alpha/\sin \Theta)4\tan^2(\Theta/2) \to \infty$ as $R \to \infty$, that is, as $\Theta \to \pi$.

Let $\Theta \to -\infty$. A modified proof of the limiting case $R \to 1$ and $\Omega = \D$ in \cite[Proposition 9.(i)]{LL22} will be applied. Refer there for more details on the following proof. Choose $h \in C^1([0,1])$ such that $h$ is increasing with $h = 0$ on $[0, 1/4]$ and $h = 1$ on $[3/4, 1]$. Take the subspace spanned by the $H^1$ functions $h(r) \cos j\phi$, $1 \leq j \leq k$, where $(r, \phi)$ are the polar coordinates. By \cite{LL22}, these functions are $L^2$-orthogonal to the constant and to each other, with respect to the weight $(1-R^2r^2)^{-2}$. Further recall $\{\cos(j\phi)\}$ is orthogonal within $L^2([0, 2\pi])$. Using this subspace as a trial function space in \eqref{finalRobinRQ}, and by homogeneity of the Rayleigh quotient, we deduce the upper bound (recalling that $\tanh(|\Theta|/2) = R$)
\begin{multline*}
0 \leq \lambda_k(\Theta, \alpha/ \sinh |\Theta|) 4\tanh^2(\Theta/2) \\ \leq \max_{|c| = 1} \frac{\displaystyle \sum_{j = 1}^k c_j^2  \int_\D |\grad (h(r)\cos j\phi)|^2 \, dA + \alpha \sum_{j = 1}^k c_j^2 \int_0^{2 \pi} |\cos j \phi|^2 \, d \phi}{\displaystyle \sum_{j = 1}^k c_j^2 \int_\D |h(r) \cos j \phi|^2 (1-R^2r^2)^{-2}\, dA}
\end{multline*}
where $c = (c_1, \dots, c_j)$ is the coefficient vector.
As $R \to 1$, each integral in the denominator tends to $\infty$, and the numerator is independent of $R$. It follows that $\lambda_k(\Theta, \alpha/ \sinh |\Theta|) 4\tanh^2(\Theta/2) \to 0$ when $R \to 1$ (i.e. when $\Theta \to -\infty)$.

\subsection*{Scaling factor $\si^2(\Theta)$}

Now consider the limits for Theorem \ref{alphapos}(ii). Let $\Theta \to \pi.$ The inequality $0 \leq \lambda_k(\Theta, \alpha / \sin \Theta)\sin^2(\Theta) \leq \lambda_k(\Theta) \sin^2(\Theta),$ where $\lambda_k(\Theta)$ is the $k^\text{th}$ Dirichlet eigenvalue, always holds for the positive eigenvalues. By \cite[Theorem 1]{LL22}, $\lambda_k(\Theta) \sin^2(\Theta) \to 0$ and thus so does $\lambda_k(\Theta, \alpha / \sin \Theta)\sin^2(\Theta)$.

Let $\Theta \to -\infty.$ We will apply Lemma \ref{weightedDivergesLemma} again. The variational characterization of the weighted eigenvalue $\lambda_k(\Theta, \alpha/ \sinh |\Theta|)\sinh^2(\Theta)$ (recalling that $\sinh |\Theta| = 2R/(1-R^2)$) is
\[
\lambda_k(\Theta, \alpha/\sinh|\Theta|) \sinh^2(\Theta):= \min_{\mathcal{L}} \max_{0 \neq g \in \mathcal{L}}
\frac{ \displaystyle \int_{\D} |\grad g(r, \phi)|^2\, dA + \alpha \int_0^{2\pi} g(1, \phi)^2\, d \phi}{\displaystyle \int_{\D} g(r, \phi)^2 \left(\frac{1-R^2}{1-R^2r^2}\right)^{\!\!2}\,dA}.
\]
Taking $\Theta \to -\infty$ is equivalent to taking $R \to 1$, and one sees the weights $(\frac{1-R^2}{1-R^2r^2})^2$ satisfy the hypotheses of Lemma \ref{weightedDivergesLemma}. (Although the lemma is stated for $R \to \infty$, a simple relabeling lets it work for $R \to 1$.) We conclude that $\lambda_k(\Theta, \alpha/ \sinh |\Theta|)\sinh^2(\Theta) \to \infty$.

\section{Negative eigenvalue limits for Theorem \ref{alphaneg}}
The negative eigenvalues are more delicate and require a thorough analysis. We need some stronger tools than those used in the previous section.

Dirichlet eigenvalues satisfy domain monotonicity, where shrinking the domain causes the eigenvalues to increase. Furthermore, for weighted Laplacian operators with positive eigenvalues, decreasing the weight also increases the eigenvalues. For negative Robin eigenvalues, these phenomena tend to \textit{reverse}. That is, the Robin eigenvalues have a tendency to \textit{decrease} (become more negative) for these processes, letting us derive tractable lower bounds on the eigenvalues.

To be precise, we consider the weighted eigenvalue problem

\begin{gather}\label{robneusys}
    \left\{
    \begin{aligned}
        & -\lapl u = \eta w u \hspace{0.74cm} \text{in } \Omega, \\
        & \frac{\partial u}{\partial n} + \alpha u = 0 \hspace{0.8cm} \text{on } \Sigma,\\
        &\frac{\partial u}{\partial n} = 0 \hspace{1.83cm}\text{on } \partial \Omega \setminus \Sigma,\\
    \end{aligned}
    \right.
\end{gather}
where $\Omega$ is a bounded domain with piecewise smooth boundary, $\Sigma \subset \partial \Omega$ is a piecewise smooth, non-empty portion of the boundary, and $w$ is some smooth positive weight function on $\overline{\Omega}$. Then, $\eta_k(\Omega, \Sigma, \alpha, w) := \eta_k$ forms a discrete spectrum
\[
\eta_1 < \eta_2 \leq \eta_3 \leq \cdots
\]
(In the case $\Sigma = \partial \Omega$, we recover a standard weighted Robin problem.) We will typically suppress notation when it is clear from context. If $\alpha$ and $w$ are clear from context, or are unchanging, then we usually denote the eigenvalues as $\eta_k(\Omega, \Sigma)$. If the Robin boundary portions are clear from context, or are unchanging, then we usually write $\eta_k(\Omega, w)$. 

The weighted Robin-Neumann eigenvalue problem allows us to control negative eigenvalues from below by shrinking the domain or decreasing the weight:

\begin{lemma}[Reversed Robin monotonicity]\label{negrobmono}
    Fix the Robin parameter $\alpha < 0$, and let $k \geq 1$.
    
    \begin{enumerate}[(i)]
    
    \item (Domain monotonicity.) Let $\Omega_1, \Omega_2 \subset \R^n$ be bounded domains with piecewise smooth boundaries. Suppose $\Sigma_1 \subset \partial \Omega_1$ and $\Sigma_2 \subset \partial \Omega_2$ are piecewise smooth, non-empty portions of the boundaries and $w$ is positive and smooth on $\overline{\Omega}$. Assume $\Omega_1 \supset \Omega_2$ and $\Sigma_1 \subset \Sigma_2$. If $\eta_k(\Omega_1, \Sigma_1) < 0$, then 
    \[\eta_k(\Omega_2, \Sigma_2) \leq \eta_k(\Omega_1, \Sigma_1) < 0.\]
    \item (Weight monotonicity.) Let $\Omega \subset \R^n$ be a bounded domain with piecewise smooth boundary. Suppose $w_1, w_2$ are two positive smooth weight functions on $\overline{\Omega}$. Assume $w_1 \geq w_2$. If $\eta_k(\Omega, w_1) < 0$, then 
    \[\eta_k(\Omega, w_2) \leq \eta_k(\Omega, w_1) < 0.\]

    \end{enumerate}
\end{lemma}
\begin{figure}[]
\centering
\begin{tikzpicture}

    \filldraw[draw=black, fill=gray!50] plot[smooth cycle, tension=0] coordinates {
     (7, 0) (11, 1) (11, -4) (8, -2)
    };

    \draw[blue, ultra thick] (11, -4) -- (8, -2);

    \draw[black, dashed] (13, -4) -- (11, -4);
    \draw[black, dashed] (13, 1.5) -- (11,1);
    \draw[black, dashed] (13, -4) -- (13, 1.5);

    \draw[blue, ultra thick] (7, 0) -- (8, -2);

    \node at (9, -3.2) {\(\Sigma_2\)};
    \node at (9.3, -1) {\(\Omega_2\)};

    \filldraw[draw=black, fill=gray!50] plot[smooth cycle, tension=0] coordinates {
     (0, 0) (6, 1.5) (6, -4) (4, -4) (1, -2)
    };

    \draw[blue, ultra thick] (4, -4) -- (1, -2);

    \node at (2, -3.2) {\(\Sigma_1\)};
    \node at (3.3, -1) {\(\Omega_1\)};
    
\end{tikzpicture}
\caption{The domain $\Omega_1$ shrinks to $\Omega_2$, and the Robin boundary portion $\Sigma_1$ enlarges to $\Sigma_2$. By Lemma \ref{negrobmono}, the negative, mixed Robin-Neumann eigenvalues on $\Omega_1$ decrease to the corresponding ones on $\Omega_2$.}
\end{figure}
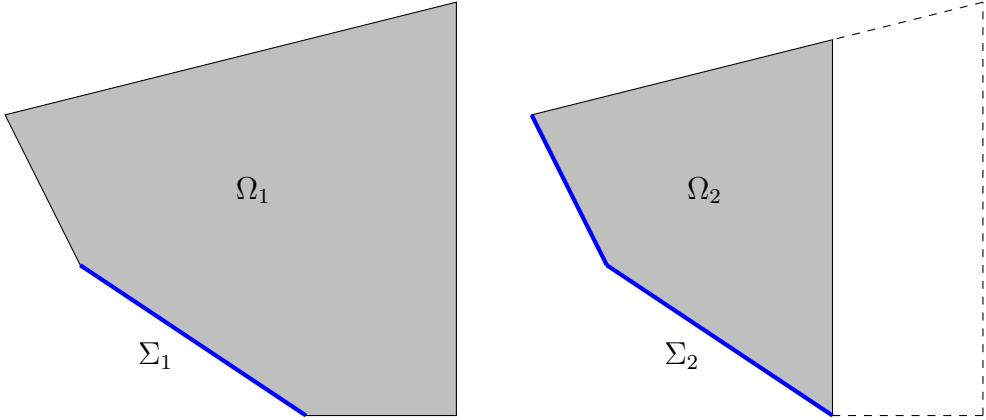
\begin{figure}[]
\centering 
\begin{tikzpicture}
    \filldraw[draw=black, fill = gray!50] (6,0) circle (2cm);
    \filldraw[draw = black, fill = white] (6,0) circle (1cm);
    \draw[blue, ultra thick] (6,0) circle (2cm);
    \node at (8.4, 0) {\(\Sigma_2\)};
    \node at (6, 1.5) {\(\Omega_2\)};

    \filldraw[draw = black, fill =gray!50] (0,0) circle (2cm);
    \draw[blue, ultra thick] (0,0) circle (2cm);
    \node at (2.4,0) {\(\Sigma_1\)};
    \node at (0, 0) {\(\Omega_1\)};
\end{tikzpicture}
\caption{Here $\Omega_1 \supset \Omega_2$ and $\Sigma_1 = \Sigma_2$. By Lemma \ref{negrobmono}, the negative eigenvalues for the pure Robin problem on the disk $\Omega_1$ are greater than or equal to the corresponding Robin-Neumann eigenvalues on the annulus $\Omega_2$. }
\label{annulusvscircle}
\end{figure}
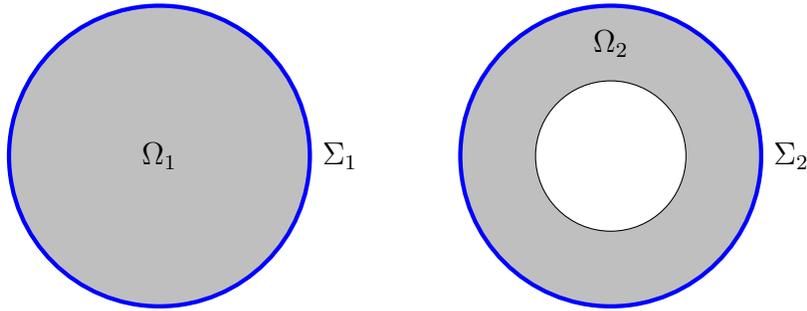
\begin{proof}
We will focus on (i) first.
The variational characterization of \eqref{robneusys} reads
\[
\eta_k(\Omega, \Sigma) = \min_{\mathcal{L}} \max_{0 \neq f \in \mathcal{L}} \frac{\displaystyle \int_\Omega |\grad f|^2 \, dA + \alpha \int_\Sigma |f|^2 \, dS}{\displaystyle \int_\Omega |f|^2 w \,dA}
\]
where $\mathcal{L}$ varies over $k$-dimensional subspaces of $H^1(\Omega)$. Suppose $\eta_k(\Omega_1, \Sigma_1) < 0$. We allow $\mathcal{L}$ to vary only over $k$-dimensional subspaces of $H^1(\Omega_1)$ such that the quantity
\[
\displaystyle \int_{\Omega_1} |\grad f|^2\, dA + \alpha \int_{\Sigma_1} |f|^2 \, dS
\]
is negative for all $f \in \mathcal{L}\setminus \{0\}$. These subspaces must exist since the eigenvalue is negative. Call the class of such subspaces $\mathfrak{N}$. It follows immediately that if $\Omega_1 \supset \Omega_2$ and $\Sigma_1 \subset \Sigma_2$ then
\begin{align*}
    \eta_k(\Omega_1, \Sigma_1) &= 
\min_{\mathcal L \in \mathfrak{N}} \max_{0 \neq f \in \mathcal{L}} \frac{\displaystyle \int_{\Omega_1} |\grad f|^2 \, dA + \alpha \int_{\Sigma_1} |f|^2 \, dS}{\displaystyle \int_{\Omega_1} |f|^2 w \,dA} \\ &\geq \min_{\mathcal L \in \mathfrak{N}} \max_{0 \neq f \in \mathcal{L}}  \frac{\displaystyle \int_{\Omega_2} |\grad f|^2  \,dA + \alpha \int_{\Sigma_2} |f|^2  \,dS}{\displaystyle \int_{\Omega_2} |f|^2 w \,dA} \geq \eta_k(\Omega_2, \Sigma_2),
\end{align*}
since the denominator is being made less positive and the numerator more negative. 

The proof of (ii) is similar: one restricts to negative numerators, and then the integrals decrease as the weight is decreased.
\end{proof}
The above lemma gives a tractable lower bound on negative Robin eigenvalues (perhaps with a weight) by taking $\Sigma_1 = \Sigma_2 = \partial\Omega_2$ and then comparing Robin eigenvalues of the disk $\Omega_1$ with Robin-Neumann eigenvalues of an annulus $\Omega_2$ as in Figure \ref{annulusvscircle}. 
\subsection*{Scaling factor $4\ta^2(\Theta/2)$}

Let $\Theta \to \pi$. We will show $\lambda_k(\Theta, \alpha/\sin \Theta)4\tan^2(\Theta/2) \to -\infty.$ Since the eigenvalue is negative for all $\Theta$ by Lemma \ref{unifpos}, it is negative at $\Theta = 0$ with eigenfunction span $\{f_1, \dots, f_k\}$. Computing \eqref{finalRobinRQ}, we have the upper bound
\[
\lambda_k(\Theta, \alpha/\sin \Theta)4\tan^2(\Theta/2) \leq \max_{f \in \text{span}\{f_1, \dots, f_k\}} \frac{\displaystyle \int_\D |\grad f|^2\, dA + \alpha \int_{\partial \D} f^2\, dS}{\displaystyle \int_\D |f|^2 (1+R^2r^2)^{-2} \,dA}.
\]
The numerator is negative for all linear combinations $f$ by negativity of the eigenvalue at $R = 0$. Additionally, the only dependency of this expression on $R$ is in the denominator, and tending $R \to \infty$ (i.e., $\Theta \to \pi$) causes the denominator to decrease to 0 for all $r > 0$. Thus, the left-hand side must go to $-\infty$.

Now let $\Theta \to -\infty$. To prove $\lambda_k(\Theta, \alpha/\sinh|\Theta|)4\tanh^2(\Theta/2) \to 0$, it is sufficient to show this is the case for the first eigenvalue. Since the Robin ground state is radial on the geodesic disk, the variational characterization \eqref{finalRobinRQ} can be written (recalling $R = \tanh(|\Theta|/2))$
\[
\lambda_1(\Theta, \alpha/\sinh|\Theta|)4\tanh^2(\Theta/2) = \min_{g \in C^1([0,1])} \frac{\displaystyle \int_0^1 g'(r)^2 r\, dr +  \alpha g(1)^2}{\displaystyle \int_0^1 g(r)^2 (1-R^2r^2)^{-2}r\, dr}
\]
where $g$ varies over the space of $C^1$ functions on $[0,1]$.
We may assume the numerator is negative since the eigenvalue is. From Lemma \ref{negrobmono}, we may shrink the domain from $[0,1]$ to $[R,1]$ and decrease the variational characterization. We thus see that
\[
\lambda_1(\Theta, \alpha/\sinh|\Theta|)4\tanh^2(\Theta/2) \geq \min_{g \in C^1([R,1])} \frac{\displaystyle \int_R^1 g'(r)^2 r\, dr +  \alpha g(1)^2}{\displaystyle \int_R^1 g(r)^2 (1-R^2r^2)^{-2}r\, dr}.
\]
By making the numerator more negative and the denominator less positive, we further deduce
\[
\frac{\displaystyle \int_R^1 g'(r)^2 r\, dr + \alpha g(1)^2}{\displaystyle \int_R^1 g(r)^2 (1-R^2r^2)^{-2}r\, dr} \geq \frac{\displaystyle \int_R^1 g'(r)^2 R\, dr +  \alpha g(1)^2}{\displaystyle \int_R^1 g(r)^2 (1-R^4)^{-2}R\, dr},
\]
where we used that $r \geq R$ and $(1-R^2r^2)^{-2} \geq (1-R^4)^{-2}$ for $r \in [R,1]$.
After performing the linear rescaling $r \mapsto \frac{r-1}{R-1}$ and doing some algebra on the $R$-terms, we conclude that
\[
\lambda_1(\Theta, \alpha/\sinh|\Theta|)4\tanh^2(\Theta/2) \geq \frac{(1-R^4)^2}{(1-R)^2 } \min_{g \in C^1([0,1])} \frac{\displaystyle \int_0^1 g'(r)^2 \, dr + \frac{(1-R)}{R} \alpha g(0)^2}{\displaystyle \int_0^1 g(r)^2  dr}.
\]
The variational characterization on the right-hand side represents the lowest eigenvalue of a mixed Robin-Neumann problem on the interval $[0,1]$ with Robin parameter $\frac{(1-R)\alpha}{R}$. By letting $R \to 1$ (which is $\Theta \to -\infty)$, the parameter tends to 0 and the eigenvalue converges to the lowest pure Neumann eigenvalue, which is $0$. (One can justify that this limit is accurate by bounding below by a pure Robin problem, or by formally letting $R \to 1$ in the Rayleigh quotient.) Since $\frac{(1-R^4)^2}{(1-R)^2} \to 16$, we conclude that $\lambda_1(\Theta, \alpha/\sinh |\Theta|)4\tanh^2(\Theta/2) \to 0$ and the claim follows.

\subsection*{Scaling factor $\si^2(\Theta)$}

The case $\Theta \to \pi$ is long and involved, and we place it in the next section

Let $\Theta \to -\infty$. To see that $\lambda_k(\Theta, \alpha/\sinh|\Theta|)\sinh^2(\Theta) \to -\infty$, let $f_1, \dots, f_k$ be the first $k$ Robin eigenfunctions for the Euclidean unit disk ($\Theta = 0$). By \eqref{finalRobinRQ}, we have the upper bound (recalling $\sinh|\Theta| = 2R/(1-R^2))$
\[
\lambda_k(\Theta, \alpha/\sinh|\Theta|)\sinh^2(\Theta)\leq 
\max_{|c| = 1} \frac{ \,\displaystyle \sum_{j=1}^k c_j^2\left( \int_\D |\grad f_j|^2 \, dA + \alpha
 \int_{\partial \D} |f_j|^2 dS\right)}{\displaystyle \sum_{j=1}^k c_j^2  \int_\D |f_j|^2 \left( \frac{1-R^2}{1-R^2r^2} \right)^{ \!\! 2} dA},
\]
where $c = (c_1, \dots, c_k)$ is the coefficient vector. The numerator is independent of $R$ and negative, so we wish to show the quantities $\int_\D |f_j|^2 \left( \frac{1-R^2}{1-R^2r^2} \right)^{ \!\! 2} dA$ tend towards $0$ as $R \to 1$. Note that the weight is bounded above by $1$ so that each $|f_j|^2 \left(\frac{1-R^2}{1-R^2r^2}\right)^2$ is bounded above by the integrable function $|f_j|^2$, and the weights tend pointwise to 0 as $R\to 1$. The claim follows from the dominated convergence theorem.

\section{Exotic limit towards exterior eigenvalue}

Consider the scaled spherical eigenvalue $\lambda_k(\Theta, \alpha/\sin\Theta)\sin^2(\Theta)$ and let $\Theta \to \pi$. In this case, we are claiming the scaled eigenvalue approaches the eigenvalue for a Robin exterior problem. Namely, we define the eigenvalue problem
\begin{equation}\label{exteriorpde}
    \left\{
    \begin{aligned}
        & -\lapl u = \lambda_k^\text{ext}(\D, \alpha)u  \quad \text{in } \mathbb{R}^2 \setminus \mathbb{D}, \\
        & \partial_{\nu} u - \alpha u = 0 \hspace{1.62cm} \text{on }  \partial\mathbb{D},
    \end{aligned}
    \right.
\end{equation}
where $\partial_v$ is the \textit{inner} normal derivative, i.e.,\ pointing into $\R^2 \setminus \D$. We call $\lambda_k^\text{ext}(\D, \alpha)$ the $k^\text{th}$ exterior Robin eigenvalue, and we claim for negative eigenvalues that 
\[\lambda_k(\Theta, \alpha/\sin \Theta)\sin^2(\Theta) \to \lambda_k^\text{ext}(\D, \alpha)\]
as $\Theta \to \pi$.

At first glance, the eigenvalue problem does not seem well-posed since the embedding $H^1(\R^2\setminus\D)$ into $L^2(\R^2 \setminus \D)$ is not compact. Indeed, the problem has essential spectrum equal to $[0, \infty)$ by \cite{KL18, KL20}. The negative eigenvalues, however, \textit{are} discrete and satisfy a variational characterization by \cite[Proposition 2.1]{KL23}. If there are $k$ negative eigenvalues, then
\begin{equation}\label{exteriorrq}
\lambda_k^\text{ext}(\D, \alpha) = \min_\mathcal{L} \max_{0 \neq f \in \mathcal{L}} \frac{\displaystyle \int_{\R^2 \setminus \D} |\grad f|^2\, dA + \alpha \int_{\partial \D} |f|^2\, dS}{\displaystyle \int_{\R^2 \setminus \D} |f|^2\, dA}
\end{equation}
where $\mathcal{L}$ varies over $k$-dimensional subspaces of $H^1(\R^2\setminus\D)$.

Consider the variational characterization \eqref{finalRobinRQ} of the spherical eigenvalue with scaling factor $\sin^2(\Theta) = 4R^2/(1+R^2)^2$. After one performs the conformal map $z \mapsto 1/z$ on the Rayleigh quotient in \eqref{finalRobinRQ}, the variational characterization reads
\begin{equation}\label{limitingRQ}
\lambda_k(\Theta, \alpha/\sin\Theta)\sin^2(\Theta)= \min_\mathcal{L}\max_{0 \neq f \in \mathcal{L}} \frac{\displaystyle \int_{\R^2 \setminus \D} |\grad f|^2\, dA + \alpha \int_{\partial \D} |f|^2\, dS}{\displaystyle \int_{\R^2 \setminus \D} |f|^2 \left( \frac{1+R^2}{r^2+R^2} \right)^{\!\! 2}\, dA}
\end{equation}
where $\mathcal{L}$ now varies over $k$-dimensional subspaces of the function space
\begin{equation}\label{hextspace}
H_\text{ext}(\R^2 \setminus \D) := \left\{ f : \int_{\R^2 \setminus \D} |f|^2 |z|^{-4} \, dA + \int_{\R^2 \setminus \D} |\grad f|^2 \, dA < \infty\right\}.
\end{equation}
The function space $H_\text{ext}(\R^2 \setminus \D)$ is \textit{not} the same as $H^1(\R^2 \setminus \D)$.  Note  that  $H^1(\R^2 \setminus \D) \subset H_\text{ext}(\R^2 \setminus \D)$.

Formally, if one takes the variational characterization $\eqref{limitingRQ}$ and lets $R \to \infty$, without regard for interchanging limits or different function spaces, the variational characterization \eqref{exteriorrq} is recovered. Of course, we \textit{do} care about interchanging limits and different function spaces, and so we must proceed quite delicately.

There are quite a few steps to the proof, and so we give a brief overview of the attack plan:

\begin{enumerate}[(i)]
    \item Show that the number of discrete negative eigenvalues is the same for both problems, so that the convergence makes sense. \label{stepone}\\
    \item Bound the spherical eigenvalues above by the exterior eigenvalues, and below by the eigenvalues of a problem related to the exterior problem, namely a Robin-Neumann problem on a large annulus.\label{steptwo}\\
    \item Justify that separation of variables applies to the exterior and annular eigenvalues, and reduce the problem to convergence of their respective radial modes.\label{stepthree}\\
    \item Prove the radial modes of the annular eigenvalues converge to the radial modes of the exterior eigenvalues as the outer radius of the annulus goes to $\infty$, by analyzing the implicit eigenvalue equations. \label{stepfour}
\end{enumerate}
Each step is proven in their respectively numbered subsections: \ref{steponeproof}, \ref{steptwoproof}, \ref{stepthreeproof}, \ref{stepfourproof}. 

\subsection{Step (\ref{stepone}): counting negative eigenvalues}\label{steponeproof}
We first show that the dimension of the negative eigenspace is equal for both the exterior problem and the spherical cap problem, and thus there is a variational characterization to converge to as $R \to \infty.$ We computed the dimension of the negative eigenspace for the spherical cap previously in Lemma \ref{unifpos}. The following lemma shows the dimension remains the same for the exterior problem.

\begin{lemma}[Counting exterior eigenvalues]\label{equalNegDim}
   If $\alpha \in [-m-1, -m)$ where $m$ is a non-negative integer, then there exist $2m+1$ negative discrete exterior eigenvalues $\lambda_k^\text{ext}(\D, \alpha)$, $k = 1, \dots, 2m+1.$
\end{lemma}
\begin{proof}
Our proof will be a modified version of \cite[Proposition 2.4]{KL23}. Denote the exterior Laplacian with parameter $\alpha$ as $-\lapl_\alpha^\text{ext}$. By \cite[Section 2.2]{KL23}, $-\lapl_\alpha^\text{ext}$ is diagonalized into an orthogonal sum of fiber operators $\mathsf{H}_{\alpha, j}$ acting on $L^2((1, \infty), r \,dr )$ for each $j \geq 0$. Each $\mathsf{H}_{\alpha, j}$ is of the form
\begin{gather*}
\mathsf{H}_{\alpha, j} f:= -f''(r) - \frac{f'(r)}{r} + \frac{j^2f(r)}{r^2}, \\
\text{dom } \mathsf{H}_{\alpha, j} = \left\{f : f, f'' + \frac{f'}{r} \in L^2((1,\infty), r \, dr), f'(1) = \alpha f(1)\right\}.
\end{gather*}
Let $\lambda_1(\mathsf{H}_{\alpha, j})$ be the infimum of the spectrum of $\mathsf{H}_{\alpha, j}$. By \cite{KL23}, there exists a constant $j_*(\alpha)$ such that $\lambda_1(\mathsf{H}_{\alpha, j})$ is negative and a discrete eigenvalue of $-\lapl_\alpha^\text{ext}$ if and only if $0 \leq j \leq j_*(\alpha)$. Furthermore, by \cite{KL23}, each $\mathsf{H}_{\alpha, j}$ has at most one negative eigenvalue, $\lambda_1(\mathsf{H}_{\alpha, 0})$ is a simple eigenvalue of $-\lapl_\alpha^\text{ext},$ and $\lambda_1(\mathsf{H}_{\alpha, j})$ is a double eigenvalue for $1 \leq j \leq j*(\alpha)$. The exterior problem thus has exactly $2j_*(\alpha) + 1$ negative eigenvalues and so our goal is to show $j_*(\alpha) = m$.

If $\alpha < 0$, the two eigenvalues $\lambda_1^\text{ext}(\D, \alpha)$ and $\lambda_1(\Theta, \alpha/ \sin \Theta)$ are both simple and negative. Our goal is then to detect when the exterior problem begins to accrue further negative eigenvalues. Let $j \geq 1$. Modifying \cite[Equation 2.11]{KL23}, the operator $\mathsf{H}_{\alpha, j}$ has a negative eigenvalue if and only if the system
\begin{equation*}
    \left\{
    \begin{aligned}
        & -f''(r) - \frac{f'(r)}{r} + \frac{j^2 f(r)}{r^2} = \sigma f(r)  \quad \text{in } (1, \infty), \\
        & f'(1) = \alpha f(1),
    \end{aligned}
    \right.
\end{equation*}
has a non-trivial eigenfunction in $L^2((1, \infty), r \, dr)$ with eigenvalue $\sigma < 0$. Using \cite[§10.25]{DLMF} and the fact an eigenfunction of the above system belongs to $L^2((1,\infty), r\,dr)$, we obtain the solution
\[
f(r) = K_j(\mu r)
\]
where $K_j$ is a modified Bessel function of the second kind and $\mu := \sqrt{-\sigma}$. Thus, $\sigma$ is a discrete negative eigenvalue of $\mathsf{H}_{\alpha, j}$ if and only if $\mu$ satisfies the Robin boundary condition at $r = 1:$
\[
\mu K_j'(\mu) = \alpha K_j(\mu).
\] 
By using the recurrence $K_j'(z) = -K_{j-1}(z) - \frac{j}{z}K_j(z)$ \cite[10.29]{DLMF}, we may rewrite the equation as
\begin{equation}\label{omegaBoundEq}
\frac{-\mu K_{j-1}(\mu)}{K_j(\mu)} = \alpha + j.
\end{equation}
Define
\[
h_j(x) := \frac{K_{j-1}(x)}{K_j(x)}
\]
so that \eqref{omegaBoundEq} reads $-xh_j(x) = \alpha + j$. By \cite[Theorem 1]{S11} and \cite[§10.37]{DLMF} we have the bounds
\begin{equation*}\label{besselBounds}
\frac{x}{j - 1/2 + \sqrt{(j- 1/2)^2 + x^2}} < h_j(x) \leq 1
\end{equation*}
    for $x > 0, j \geq 1$. It follows that $x h_j(x)$ tends to $0$ as $x \to 0$ and tends to $\infty$ as $x \to \infty$. We want to show that $xh_j(x)$ is strictly increasing so that the solution of \eqref{omegaBoundEq} is unique. From \cite[Lemma 1]{S11}, we see $h_j(x)$ is strictly increasing in a neighborhood of the origin. The proof of \cite[Theorem 2]{S11} shows $h_j'(x) > 0$ for all $x > 0$. It follows that $xh_j(x)$ is strictly increasing.

We conclude \eqref{omegaBoundEq} has a unique positive solution $\mu$ if and only if $\alpha + j < 0$. Since $\alpha \in [-m-1, -m),$ the condition $\alpha + j < 0$ means $j \leq m$ and so $j_*(\alpha) = m$. Hence, the dimension of the negative eigenvalues on the spherical cap and on the exterior of the unit disk are equal in the case $\alpha \in [-m-1, -m)$. 
\end{proof}

We have managed to show step (\ref{stepone}), justifying that the spherical eigenvalues \textit{can} converge to the exterior ones. Now, we must show that they \textit{do} converge. We start with some upper and lower bounds.

\subsection{Step (\ref{steptwo}): upper and lower bounds.}\label{steptwoproof}
Henceforth, we assume $k \leq 2m+1$ so that $\lambda_k(\Theta, \alpha)$ and $\lambda_k^\text{ext}(\D, \alpha)$ are both negative. Perform the conformal map $z \mapsto 1/z$ on \eqref{finalRobinRQ} to yield the characterization \eqref{limitingRQ} of $\lambda_k(\Theta, \alpha/\sin \Theta)\sin^2(\Theta)$. Observe in \eqref{limitingRQ} that the weight in the denominator is bounded by 1, and we may assume the numerator is negative. Furthermore, we may restrict $\mathcal{L}$ to vary over $k$-dimensional subspaces of $H^1(\R^2 \setminus \D)$ rather than $H_\text{ext}(\R^2 \setminus \D)$ (defined in \eqref{hextspace}), in light of the inclusion $H^1(\R^2 \setminus \D) \subset H_\text{ext}(\R^2 \setminus \D)$. Hence, we obtain  the upper bound
\begin{equation}\label{interupperBound}
\lim_{\Theta \to \pi} \lambda_k(\Theta, \alpha/\sin \Theta)\sin^2(\Theta) \leq \lambda_k^\text{ext}(\D, \alpha) < 0.
\end{equation}

We cannot directly bound below by the exterior eigenvalues nor take the limit of \eqref{limitingRQ} as $R \to \infty$. Instead, we will chop off the unbounded portion of the domain and consider a large annulus with a Robin boundary on the inner radius and Neumann boundary on the outer annulus. We are able to bound the scaled spherical eigenvalue below by these \textit{annular} eigenvalues.  

Define $A(M)$ as the annulus of inner radius $1$ (the Robin portion) and outer radius $M$ (the Neumann portion), and define $w$ as the weight $\left(\frac{1+R^2}{r^2+R^2}\right)^{ \!\! 2}$.
By reworking Lemma \ref{negrobmono} slightly for unbounded domains and different function spaces, we see that
\begin{align*}
\lambda_k(\Theta, \alpha/\sin\Theta)\sin^2(\Theta) & = \min_\mathcal{L}\max_{0 \neq f \in \mathcal{L}} \frac{\displaystyle \int_{\R^2 \setminus \D} |\grad f|^2\, dA + \alpha \int_{\partial \D} |f|^2\, dS}{\displaystyle \int_{\R^2 \setminus \D} |f|^2 \left( \frac{1+R^2}{r^2+R^2} \right)^{\!\! 2}\, dA} \\
& \geq
\min_\mathcal{L}\max_{0 \neq f \in \mathcal{L}} \frac{\displaystyle \int_{A(M)} |\grad f|^2\, dA + \alpha \int_{\partial \D} |f|^2\, dS}{\displaystyle \int_{A(M)} |f|^2 \left( \frac{1+R^2}{r^2+R^2} \right)^{\!\! 2}\, dA}
\end{align*}
where $\mathcal{L}$ varies over $k$-dimensional subspaces of $H_\text{ext}(\R^2 \setminus \D)$. The key point here is that if $f$ is in $H_\text{ext}(\R^2 \setminus \D)$, the restriction of the function to $A(M)$ must be in $L^2(A(M))$ since $|z|^4$ is a bounded function on the annulus. That is, when we restrict the function space to exist on $A(M)$, we can revert to the standard $H^1$ space.
We thus generate the lower bound
\[
\lambda_k(\Theta, \alpha/ \sin \Theta) \sin^2(\Theta) \geq \eta_k(A(M), \partial \D, \alpha, w).
\]
Since the weight $w$ converges uniformly to $1$ on the annulus, taking $R \to \infty$ yields the bound
\begin{equation}\label{interlowerBound}
\lim_{\Theta \to \pi} \lambda_k(\Theta, \alpha/\sin \Theta) \sin^2(\Theta) \geq \eta_k(A(M), \partial \D,  \alpha, 1).
\end{equation}
For the right-hand side, we will suppress notation and denote the unweighted Robin-Neumann eigenvalue on the annulus as $\eta_k(M, \alpha)$. Take care to note the inner radius at $R = 1$ has the Robin boundary condition for this eigenvalue problem. Again, using Lemma \ref{negrobmono}, this eigenvalue is an increasing function in $M$. By combining \eqref{interupperBound} and \eqref{interlowerBound}, we see the limit $\lim_{M \to \infty} \eta_k(M, \alpha)$ is a number $\eta_k^*(\alpha) < 0$. We will show $\eta_k^*(\alpha) = \lambda_k^\text{ext}(\D, \alpha)$.

Step (\ref{steptwo}) is now complete. To analyze further, we will look at the implicit equation that characterizes the exterior eigenvalues and the one that characterizes the annular eigenvalues. In order to do this, we will separate variables and reduce our analysis to the radial modes. The radial modes satisfy certain well-behaved implicit equations characterizing the eigenvalues completely, and we will show uniform convergence of these equations.

\subsection{Step (\ref{stepthree}): reduction to the radial modes}\label{stepthreeproof}

First, we state properties of the exterior eigenvalues. Observe the domain $\R^2 \setminus \D$ is radially symmetric. Then, referring to Subsection \ref{steponeproof}, we separate variables into polar coordinates and find the angular components of the modes are $\cos(j\phi), \sin(j\phi)$ and the radial components are found by solving the modified Bessel ODE system

\begin{equation}\label{separatedExteriorODE}
    \left\{
    \begin{aligned}
        & -f''(r) - \frac{f'(r)}{r} + \frac{j^2 f(r)}{r^2} = \sigma f(r)  \quad \text{in } (1, \infty), \\
        & f'(1) = \alpha f(1). \\
    \end{aligned}
    \right.
\end{equation}
For $\alpha \in [-m-1, -m)$, the system \eqref{separatedExteriorODE} has a unique eigenvalue $\sigma_j(\alpha) < 0$ for each $0 \leq j \leq m$ from the proof of Lemma \ref{equalNegDim}, and by \cite[Section 2.2]{KL23}, we have that $\sigma_k(\alpha) < \sigma_j(\alpha) < 0$ if $0 \leq k < j \leq m$.  The eigenvalue $\sigma_j(\alpha)$ has associated eigenfunction
\[
K_j(\mu_j(\alpha) r)
\]
where $\mu_j(\alpha) := \sqrt{-\sigma_j(\alpha)}$. So, separation of variables is quite quickly seen to be justified in the exterior case.

For annular eigenvalues, there is more work to be done, since less literature exists on this subject. Freitas and Krejčiřík provide a treatment of the first Robin eigenvalue with negative parameter on disks and annuli in \cite{FK15}. Here, we will be considering all negative eigenvalues, including the first. 

Rotational symmetry of the annulus again allows us to separate variables into polar coordinates, for the mixed Robin-Neumann problem. The angular components are of the form $\cos(j \phi)$ and $\sin (j \phi)$ while the radial components are found by solving the modified Bessel ODE system

\begin{equation}\label{separatedAnnularODE}
    \left\{
    \begin{aligned}
        & -f''(r) - \frac{f'(r)}{r} + \frac{j^2 f(r)}{r^2} = \rho f(r)  \quad \text{in } (1, M), \\
        & f'(1) = \alpha f(1), \\
        & f'(M) = 0.
    \end{aligned}
    \right.
\end{equation}
We need to prove for each $j = 1, \dots, m$ that the system possesses a unique eigenvalue $\rho_j(M, \alpha) < 0$ satisfying the eigenfunction equation and boundary conditions. Furthermore, we also want to prove that increasing $j$ increases the eigenvalue $\rho_j(M,\alpha)$. After proving these properties, we can justify separation of variables in the annular case.

\begin{lemma}[Negative Robin eigenvalues on annulus]\label{annularRobinEigs}
    Let $\alpha \in [-m-1, -m)$ where $m$ is a non-negative integer. If $0 \leq j \leq m$, then there exists a unique eigenvalue $\rho_j(M,\alpha) < 0$ of system \eqref{separatedAnnularODE}. Furthermore, if $0 \leq k < j \leq m$, then $\rho_k(M,\alpha) < \rho_j(M, \alpha) < 0$.
\end{lemma}
\begin{proof}
System \eqref{separatedAnnularODE} induces a discrete spectrum of eigenvalues
\[
\rho_{j,1}(M, \alpha) < \rho_{j,2}(M,\alpha) \leq \rho_{j,3}(M,\alpha) \leq \cdots \to \infty
\]
and we claim $\rho_{j,1}(M,\alpha) < 0$ while $\rho_{j,2}(M,\alpha) \geq 0$. 

The variational characterization for the lowest eigenvalue reads
\begin{equation}\label{fiberRayleigh}
\rho_{j,1}(M,\alpha) = \min_{f \neq 0} \frac{\displaystyle \int_1^M \left(f'(r)^2 + \frac{j^2}{r^2}f(r)^2\right) r\, dr + \alpha f(1)^2}{\displaystyle \int_1^M f(r)^2r\, dr }
\end{equation}
where $f$ varies over $H^1([1,M], r\,dr)$. Furthermore, also consider the variational characterization for the lowest eigenvalue $\sigma_j(\alpha)$ of System \eqref{separatedExteriorODE}:
\[
\sigma_j(\alpha) = \min_{f \neq 0} \frac{\displaystyle \int_1^\infty \left(f'(r)^2 + \frac{j^2}{r^2}f(r)^2\right) r\, dr + \alpha f(1)^2}{\displaystyle \int_1^\infty f(r)^2r\, dr }
\]
where $f$ varies over $H^1([1, \infty], r\, dr).$ 
As $\alpha \in [-m-1, -m)$ by assumption, the eigenvalue $\sigma_j(\alpha)$ is negative from the proof of Lemma \ref{equalNegDim}, so we may restrict its variational characterization to functions $f$ such that the numerator is negative. By reducing the domain of integration in the numerator and denominator (such as in Lemma \ref{negrobmono}), we see immediately that
\[
\rho_{j,1}(M, \alpha) \leq \sigma_j(\alpha) < 0.
\]

We show non-negativity for the rest of the spectrum. Let $f_1$ and $f_2$ be eigenfunctions for $\rho_{j,1}(M,\alpha)$ and $\rho_{j,2}(M,\alpha)$ respectively. Since $f_1$ cannot change sign on $[1,M]$ and $f_2$ is orthogonal to $f_1$ in $L^2([1, M], r\,dr)$, we see that $f_2$ must change sign at some point $x^* \in (1,M)$. Then, the restriction of $f_2$ to the interval $[x^*, M]$ is an eigenfunction of the system 
\begin{equation*}
    \left\{
    \begin{aligned}
        & -f''(r) - \frac{f'(r)}{r} + \frac{j^2 f(r)}{r^2} = \rho f(r)  \quad \text{in } (x^*, M), \\
        & f(x^*) = 0, \\
        & f'(M) = 0.
    \end{aligned}
    \right. 
\end{equation*} 
The system above represents an eigenvalue problem for a mixed Dirichlet-Neumann operator on the interval $[x^*, M]$. The variational characterization for this system implies the operator's eigenvalues must be non-negative, and it follows that $\rho_{j,2}(M,\alpha) \geq 0$. Hence, $\rho_{j,1}(M, \alpha)$ is the unique negative eigenvalue of \eqref{separatedAnnularODE}, and we denote it $\rho_j(M, \alpha)$.

By comparing the variational characterizations \eqref{fiberRayleigh}, we see $\rho_k(M, \alpha) < \rho_j(M,\alpha)$ if $k < j$ and the claim follows.
\end{proof}
Note that Lemma \ref{annularRobinEigs} is \textit{not} stating there are exactly $2m+1$ negative eigenvalues, rather that there are \textit{at least} $2m+1$ negative eigenvalues. 

For the exterior and annular radial modes, we see that increasing the parameter $j$ increases both eigenvalues $\sigma_j(\alpha)$ and $\rho_j(M, \alpha)$. Furthermore, for $j \geq 1$, each of the eigenvalues is repeated twice in the unseparated form of their respective PDE systems due to the existence of the two angular components $\cos(j\phi)$ and $\sin(j\phi)$. These facts, along with $\sigma_j(\alpha)$ and $ \rho_j(M,\alpha)$ being unique negative eigenvalues for the separated Bessel ODEs, imply that
\begin{align*}
\sigma_j(\alpha) &= \lambda_{2j}^\text{ext}(\D, \alpha) = \lambda_{2j+1}^\text{ext}(\D, \alpha) , \\  \rho_j(M, \alpha) &= \eta_{2j} (M, \alpha) = \eta_{2j+1}(M, \alpha),
\end{align*}
for $1 \leq j \leq m$, where we recall $\lambda_k^\text{ext}(\D, \alpha)$ is the $k^\text{th}$ exterior eigenvalue and $\eta_k(M,\alpha)$ is the $k^\text{th}$ Robin-Neumann eigenvalue on the annulus $A(M)$. For $j = 0$, the exterior and annular eigenvalues are simple and radial so that they coincide with their radial modes. Thus, proving $\lim_{M\to\infty} \eta_k(M, \alpha) = \lambda_k^\text{ext}(\D, \alpha)$, $0 \leq k \leq 2m+1$, is equivalent to proving the radial modes converge:
\[
\lim_{M \to \infty} \rho_j(M, \alpha) = \sigma_j(\alpha), \quad 0 \leq j \leq m.
\]

We have successfully completed step (\ref{stepthree}) of the proof. We will finish the final step of the proof by showing the implicit equations corresponding to the radial eigenvalues converge uniformly.

\subsection{Step (\ref{stepfour}): convergence of the radial modes}\label{stepfourproof} By Subsection \ref{steponeproof}, $\mu_j(\alpha) = \sqrt{-\sigma_j(\alpha)}$ is the unique root in $(0, \infty)$ of the implicit function
\[
A(x, \alpha) := x K_j'(x) - \alpha K_j(x).
\]
Similarly to how Equation (13) is derived in \cite{FK15}, we find $\omega_j(M, \alpha) := \sqrt{-\rho_j(M, \alpha)}$ is the unique root in $(0, \infty)$ of the implicit function
\[
G_M(x, \alpha) := A(x, \alpha)D_M(x, \alpha) - B(x, \alpha)C_M(x, \alpha)
\]
with coefficients given in terms of the modified Bessel functions $I_j$ and $K_j$ as
\begin{align*}
    A(x, \alpha) &:= x K_j'(x) - \alpha K_j(x), \\
    B(x, \alpha) &:= x I_j'(x) - \alpha I_j(x), \\
    C_M(x, \alpha) & := x K_j'(Mx), \\
    D_M(x, \alpha) & := x I_j'(Mx).
\end{align*}
(Despite the lack of dependence of $C_M$ and $D_M$ on $\alpha$, we keep the notation for consistency.)

Since $D_M$ is non-vanishing on $(0, \infty)$, we may divide $G_M$ by $D_M$ and preserve the location of the unique root of $G_M$. Furthermore, recall that $\eta_k(M,\alpha)$ is increasing as a function of $M$ and bounded above by $\lambda_k^\text{ext}(\D, \alpha)$ (see the paragraph after Equation \eqref{interupperBound}). It follows that $\omega_j(M, \alpha)$ is decreasing as a function of $M$ and bounded below by $\mu_j(\alpha) > 0$, and so we may restrict our analysis to the compact interval 
\[
S = [\mu_j(\alpha)/2, \omega_j(M_0, \alpha)]
\]where $M_0$ is large and finite. We will show that
$G_M(x,\alpha)/D_M(x, \alpha)$ converges uniformly to $A(x,\alpha)$ on the interval $S$ as $M \to \infty$, or, equivalently:
\[
-\frac{B(x, \alpha)C_M(x, \alpha)}{D_M(x, \alpha)} \xrightarrow{\text{unif.}} 0 \quad \text{on } S \text{ as }M \to \infty.
\]
Since the convergence is uniform, the roots $\omega_j(M, \alpha)$ must converge to a root of $A(x,\alpha)$, and as $A(x, \alpha)$ has $\mu_j(\alpha) \in S$ as its unique root, we will be able to conclude the proof.

By \cite[§10.29]{DLMF}, we have the recurrence relations

\begin{alignat*}{3}
& I_j' = \frac{1}{2}\left(I_{j-1} + I_{j+1}\right), &&\quad -K_j' = \frac{1}{2}\left(K_{j-1} + K_{j+1}\right), &&\quad (j \geq 1)\\
& I_0' = I_1, &&\quad -K_0' = K_1, &&\quad (j = 0).
\end{alignat*}
From these relations we see that $I_j'$ increases strictly from $0$ to $\infty$ (or $1/2$ to $\infty$ if $j = 1$) and $-K_j'$ decreases strictly from $
\infty$ to $0$. Thus, it follows the ratio
\[
\frac{-C_M(x,\alpha)}{D_M(x,\alpha)} = \frac{-K_j'(Mx)}{I_j'(Mx)}
\]
is positive and decreases to 0 as $M$ increases. Furthermore, the ratio achieves its maximum value at the left endpoint of $S$ for each fixed $M$. Since $B$ is positive and does not depend on $M$, we deduce that
\[
\sup_{x \in S} \frac{-B(x,\alpha)C_M(x,\alpha)}{D_M(x,\alpha)} \leq B(\mu_j(\alpha)/2, \alpha) \frac{-K_j'(M \mu_j(\alpha)/2)}{I_j'(M \mu_j(\alpha)/2)} \to 0 \quad \text{as } M \to \infty.
\]
It follows that the convergence is uniform, and we conclude that
\[
\lim_{\Theta \to \pi} \lambda_k(\Theta, \alpha/\sin \Theta)\sin^2(\Theta) = \lim_{M \to \infty} \eta_k(M, \alpha) = \lambda_k^\text{ext}(\D, \alpha).
\]

\section{Higher-dimensional Dirichlet monotonicity --- Theorem 1.3}\label{DirichProof}
For this section, we jump up to dimensions three and higher, and the reader needs to keep in mind we now analyze higher-dimensional Dirichlet eigenvalues, not two-dimensional Robin eigenvalues. One may find it helpful to refer to Appendix \ref{coords} for a refresher on the Laplace--Beltrami operator in our local coordinate system. The techniques here are quite different from the proofs of Theorems \ref{alphapos} and \ref{alphaneg}, as we use direct trial functions instead of conformal maps. 

There are three main steps in the proof of Theorem \ref{main}:
\begin{itemize}
    \item Transform an eigenspace on $C(\Theta_1)$ to a trial subspace on $C(\Theta_2)$.

    \item Rescale the Rayleigh quotient on $C(\Theta_2)$ back to $C(\Theta_1)$, incurring the cost of certain coefficient functions.

    \item Bound the coefficient functions and recover the standard Rayleigh quotient for $\lambda_k(\Theta_1)$.
\end{itemize}
The main point of interest in the proof is in the first step, where the transformation is performed to preserve the area elements of $C(\Theta_1)$ and $C(\Theta_2)$. We will first consider the spherical case corresponding to $\Theta > 0$, as the proof is slightly more demanding there.

\subsection{Spherical case}
We first note that $\lambda_k(\Theta)$ is a positive decreasing function of $\Theta$ by Dirichlet domain monotonicity, and as $\sin^2$ decreases strictly on $[\pi/2, \pi]$ and is positive, we need only show that $\lambda_k(\Theta)\sin^2(\Theta)$ is decreasing strictly on the  interval $(0, \pi/2)$. 

For any $\Theta$ we have the variational characterization.
\[
\lambda_k(\Theta) = \min_\mathcal{L} \max_{0 \neq f \in 
\mathcal{L}} \frac{\displaystyle \int_{C(\Theta)} |\grad_{\S^n} f|^2 \, dV }{\displaystyle \int_{C(\Theta)}  |f|^2 \, dV } =: \min_\mathcal{L} \max_{0 \neq f \in 
\mathcal{L}}  \mathcal{R}_\Theta(f)
\]
where $\mathcal{L}$ varies over $k$-dimensional subspaces of $H_0^1(C(\Theta))$. In coordinates, the Rayleigh quotient takes the form
\[
\mathcal{R}_\Theta(f) =\frac{ \displaystyle\int_{\S^{n-1}}\int_0^\Theta \left(|\partial_1 f|^2 + \frac{1}{\sin^2\theta}|\grad_{\S^{n-1}} f|^2\right)\sin^{n-1}\theta \, d\theta d\xi}{ \displaystyle \int_{\S^{n-1}}\int_0^\Theta |f|^2\sin^{n-1}\theta \, d\theta d\xi},
\]
where we write $\partial_1 f$ to denote the derivative in the first component, the aperture coordinate. Since we frequently pass between different spherical caps, this notation is employed for clarity. 

Let $0 < \Theta_1 < \Theta_2 \leq \pi/2.$ Our goal is to show $\lambda_k(\Theta_1)\sin^2(\Theta_1) > \lambda_k(\Theta_2)\sin^2(\Theta_2)$. Let $\{\phi_1, \dots, \phi_k\}$ be the first $k$ Dirichlet eigenfunctions on $C(\Theta_1)$. Call the subspace generated by these functions $\mathcal{K}$. We want to extend these functions suitably to $C(\Theta_2)$ to be trial functions in the Rayleigh quotient for $\lambda_k(\Theta_2)$. It is unclear what such an extension $\tilde{\phi}_j$ should be. A simple linear extension in the $\theta$ variable would not interact well with the Rayleigh quotient. The volume element would transform into $\sin^{n-1}(\Theta_1\theta/\Theta_2)$ in both the numerator and denominator and not cancel out with any factors, crushing our hopes of proving monotonicity. However, if we normalize the trial function by the volume element, we can defeat this obstacle. We are thus motivated to define the following transformation: let $f \in H_0^1(C(\Theta_1))$ and define
\[
\tilde{f} (\theta, \xi) := f(\Theta_1 \theta/\Theta_2, \xi) \frac{\sqrt{\sin^{n-1}(\Theta_1 \theta/\Theta_2)}}{\sqrt{\sin^{n-1}(\theta)}} \in H_0^1(C(\Theta_2)).
\]
(At the origin, the ratio of sines is understood as its limit $\lim_{\theta \to 0}\sin(\Theta_1 \theta/\Theta_2)/\sin(\theta) = \Theta_1/\Theta_2$.) In essence, this transformation is the spherical equivalent of linear dilation in the plane. Linear dilation in the plane stretches a function uniformly to fit into a larger domain while only incurring a constant multiplicative cost to the $L^2$ norm. Here one sees the same phenomenon and visualizes it suitably adjusted to a spherical cap scenario. One stretches the function along the spherical cap as one might stretch rubber, while scaling the $L^2$ norm by a constant, with 
\[
\int_{C(\Theta_2)} |\tilde{f}|^2 \, dV = \frac{\Theta_2}{\Theta_1}\int_{C(\Theta_1)} |f|^2 \, dV.
\]
We will state the slightly stronger lemma:
\begin{lemma}\label{tildetransf}
Let $0 < \Theta_1 < \Theta_2 \leq \pi/2$ and $f \in H_0^1(C(\Theta_1))$. Then, as defined above, $\tilde{f} \in H_0^1(C(\Theta_2))$. Further, if $\mathcal{K} = \text{span}\{f_1, \dots, f_k\}$ is linearly independent in $H_0^1(C(\Theta_1))$, then $\tilde{\mathcal{K}} := \text{span}\{\tilde{f_1}, \dots, \tilde{f_k}\}$ is linearly independent in $H_0^1(C(\Theta_2))$, and the map $f \mapsto \tilde{f}$ is linear.
\end{lemma}
\begin{proof}
The calculations done below resulting in Equation (\ref{finalRayl}) show that $\norm{\grad_{\S^n} \tilde{f}}_{L^2(C(\Theta_2))}^2$ is bounded above by $\frac{\sin^2(\Theta_1)}{\sin^2(\Theta_2)}\norm{\tilde{f}}_{L^2(C(\Theta_2))}^2\mathcal{R}_{\Theta_1}(f)$. One can also clearly see $\norm{\tilde{f}}_{L^2(C(\Theta_2))}^2$ is finite. The linear independence and linearity are trivial and stated for completeness.
\end{proof}
Consider the trial functions on $C(\Theta_2)$ defined to be $\tilde{\phi}_j$ for $1 \leq j \leq k$, where the $\phi_j$ are Dirichlet eigenfunctions on $C(\Theta_1)$. The $\tilde{\phi}_j$ are linearly independent, so let their span be our trial subspace $\tilde{\mathcal{K}}$ in the variational characterization for $\lambda_k(\Theta_2)$. The maximizing function is thus achieved for some linear combination $\tilde{g} = \sum a_j \tilde{\phi}_j$ so that
\[
\lambda_k(\Theta_2) \leq \frac{ \displaystyle\int_{\S^{n-1}}\int_0^{\Theta_2} \left(|\partial_1 \tilde{g}|^2 + \frac{1}{\sin^2\theta}|\grad_{\S^{n-1}} \tilde{g}|^2\right)\sin^{n-1}(\theta) \, d\theta d\xi}{ \displaystyle \int_{\S^{n-1}}\int_0^{\Theta_2} |\tilde{g}|^2\sin^{n-1}(\theta) \, d\theta d\xi}.
\]
Define $g := \sum a_j \phi_j \in H_0^1(C(\Theta_1))$ and note that 
\[
|\grad_{\S^{n-1}} \tilde{g}|^2 = \frac{\sin^{n-1}(\Theta_1 \theta/ \Theta_2)}{\sin^{n-1}\theta} |\grad_{\S^{n-1}}g|^2
\]
as the transformation only affects $\theta$ and not $\xi$. After the linear change of variables $\theta = \Theta_2 t /\Theta_1$, we find an upper bound for $\lambda_k(\Theta_2)$ of

\[
\frac{ \displaystyle\int_{\S^{n-1}}\int_0^{\Theta_1} \left(|(\partial_1 \tilde{g})(\Theta_2 t/\Theta_1, \xi)|^2 \frac{\sin^{n-1}(\Theta_2 t/\Theta_1)}{\sin^{n-1}(t)}+ \frac{1}{\sin^2(\Theta_2 t/\Theta_1)}|\grad_{\S^{n-1}} g(t, \xi)|^2\right)\sin^{n-1}(t) \, dt d\xi}{ \displaystyle \int_{\S^{n-1}}\int_0^{\Theta_1} |g(t, \xi)|^2\sin^{n-1}(t) \, dt d\xi}.
\]

Our goal is to compare this last displayed formula with a Rayleigh quotient over $C(\Theta_1)$ involving the function $g$. A routine calculation shows in the numerator that
\begin{multline*}
|(\partial_1 \tilde{g})(\Theta_2 t/\Theta_1, \xi)|^2 \frac{\sin^{n-1}(\Theta_2 t/ \Theta_1)}{\sin^{n-1}(t)} = \frac{1}{4 \Theta_2^2} \Bigl( (  n-1)^2\bigl( \cot(t)\Theta_1 -  \cot ( \Theta _2 t/\Theta_1)\Theta_2 \bigr)^2g(t, \xi)^2 \\ + 4\Theta_1^2\partial_1 g(t, \xi)^2  + 4 (n-1)\Theta_1 \bigl( \cot(t) \Theta_1 - \cot (\Theta_2 t/ \Theta_1)\Theta_2 \bigr)g(t, \xi) \partial_1 g(t, \xi) \Bigr).
\end{multline*}
At this point, we would like to perform a direct comparison of Rayleigh quotients, but $g\partial_1 g$ need not be positive. So, we write $2g\partial_1 g = (g^2)_1$ and integrate by parts. The boundary terms vanish by the Dirichlet boundary conditions and the fact that $\sin^{n-2}(t) \cot (t) = 0$ when $t = 0$ and $n \geq 3$.

We thus obtain an upper bound for $\lambda_k(\Theta_2)\sin^2(\Theta_2)$ (observe here we multiply the Rayleigh quotient by the scaling factor $\sin^2(\Theta_2)$) of
\begin{equation}\label{finalRayl}
\lambda_k(\Theta_2)\sin^2(\Theta_2)\leq\frac{\displaystyle \int_{\S^{n-1}}\int_0^{\Theta_1} \left( P|\partial_1 g|^2 +  Q|\grad_{\S^{n-1}}g|^2 + R|g|^2 \right)\sin^{n-1}(t)\, dt d\xi}{\displaystyle \int_{\S^{n-1}}\int_0^{\Theta_1} |g|^2 \sin^{n-1}(t) \, dt d\xi}
\end{equation}
where the coefficient functions are
\begin{equation*}
\begin{gathered}
P :=  \sin^2(\Theta_2)\frac{\Theta_1^2}{\Theta_2^2}, \\
Q:= \frac{\sin^2(\Theta_2)}{\sin^2(\Theta_2 t / \Theta_1)}, \\
R :=   \frac{\sin^2(\Theta_2)(n-1)}{4 \Theta _2^2} \Bigl(\Theta_2^2
   ((n-1) \cot ^2(\Theta_2 t/\Theta _1)-2 \csc
   ^2(\Theta _2 t/\Theta _1)) \\ \hspace{7cm} -\Theta _1^2 ((n-1) \cot ^2(t) - 2 \csc ^2(t)) \Bigr).\\
\end{gathered}
\end{equation*}
We wish to recover the original form of the Rayleigh quotient by bounding $P, Q$, and $R$ above by the proper quantities. Namely, we want to show $P$ is bounded by $\sin^2(\Theta_1)$, $Q$ is bounded by $\sin^2(\Theta_1)/\sin^2(t)$, and $R$ is non-positive.

\subsubsection*{Bound on $P$}
As $\sin^2(x)/x^2$ is a strictly decreasing function up to $\pi$, we readily see that $P < \sin^2(\Theta_1)$.

\subsubsection*{Bound on $Q$}
Consider the function $\sin^2(x)/\sin^2(x t/\Theta_1)$ where $t < \Theta_1$. By taking a derivative in $x$, we see the function is strictly decreasing for $x \in (0, \pi/2)$ if the inequality $\tan(tx/\Theta_1) < \frac{t}{\Theta_1}\tan(x) $ holds. As $\tan(x)/x$ is a strictly increasing function, it follows that $Q < \frac{\sin^2(\Theta_1)}{\sin^2(t)}$.

\subsubsection*{Bound on $R$}
We will show that $R$ is negative. Since $\frac{\sin^2(\Theta_2)(n-1)}{4\Theta_2^2}$ is positive, we need only show negativity of the quantity within the parentheses. It is enough to show
\begin{equation}\label{cotdec}
x^2 \left((n-1) \cot ^2\left(\frac{x t}{\Theta _1}\right)-2 \csc
   ^2\left(\frac{x t}{\Theta _1}\right)\right) 
\end{equation}
is a decreasing function for $x \in (0, \pi/2)$. By differentiating with respect to $x$, we want to show
\begin{equation*}\label{cotneg}
(n-1)\cot^2y - (n-3)y\cot y\csc^2 y-2\csc^2 y
\end{equation*}
is negative for $y \in (0, \pi/2)$. Rewrite the above expression as
\[
(n-3)\csc^2(y)\left(1 - y \cot y\right) - (n-1). \
\] 
Since $y \cot y < 1$ on $(0, \pi/2)$ and $n \geq 3$, the expression is clearly negative and \eqref{cotdec} is decreasing.

With these bounds on the coefficient functions, it follows that
\[
\lambda_k(\Theta_2)\sin^2(\Theta_2) < \mathcal{R}_{\Theta_1}(g)\sin^2(\Theta_1).
\]
Since $g \in \text{span}\{\phi_1, \dots, \phi_k\}$, we know $\mathcal{R}_{\Theta_1}(g) \leq \lambda_k(\Theta_1)$
and thus $\lambda_k(\Theta)\sin^2(\Theta)$ is a strictly decreasing function on $(0, \pi/2)$. 

\subsection{Hyperbolic case}

The proof in this case is virtually identical. Hyperbolic trig functions replace their Euclidean counterparts.

For $\Theta < 0$, the Rayleigh quotient is
\[
\mathcal{R}_\Theta(f) =\frac{ \displaystyle\int_{\S^{n-1}}\int_0^{|\Theta|} \left(|\partial_1 f|^2 + \frac{1}{\sinh^2\theta}|\grad_{\S^{n-1}} f|^2\right)\sinh^{n-1}\theta \, d\theta d\xi}{ \displaystyle \int_{\S^{n-1}}\int_0^{|\Theta|} |f|^2\sinh^{n-1}\theta \, d\theta d\xi}.
\]

Let $\Theta_1 < \Theta_2 < 0$. We wish to show that $\lambda_k(\Theta_1)\sinh^2(\Theta_1) > \lambda_k(\Theta_2)\sinh^2(\Theta_2)$. We rework Lemma \ref{tildetransf} and define
\[
\tilde{f} (\theta, \xi) := f(|\Theta_1 |\theta/|\Theta_2|, \xi) \frac{\sqrt{\sinh^{n-1}(|\Theta_1| \theta/|\Theta_2|)}}{\sqrt{\sinh^{n-1}(\theta)}}.
\]
\begin{lemma}
Let $\Theta_1 < \Theta_2 < 0$ and $f \in H_0^1(C(\Theta_1))$. Then, $\tilde{f} \in H_0^1(C(\Theta_2))$ and the mapping $f \mapsto \tilde{f}$ is linear and preserves linear independence.
\end{lemma}
The proof of the theorem proceeds the same, making necessary modifications, and we arrive at the bound 
\begin{equation*}
\lambda_k(\Theta_2)\sinh^2(\Theta_2) \leq \frac{\displaystyle \int_{\S^{n-1}}\int_0^{|\Theta_1|} \left( P|\partial_1 g|^2 +  Q|\grad_{\S^{n-1}}g|^2 + R|g|^2 \right)\sinh^{n-1}(t)\, dt d\xi}{\displaystyle \int_{\S^{n-1}}\int_0^{|\Theta_1|} |g|^2 \sinh^{n-1}(t) \, dt d\xi}
\end{equation*}
where the new hyperbolic coefficient functions are
\begin{equation*}
\begin{gathered}
P :=  \sinh^2(|\Theta_2|)\frac{|\Theta_1|^2}{|\Theta_2|^2}, \\
Q:= \frac{\sinh^2(|\Theta_2|)}{\sinh^2(|\Theta_2| t / |\Theta_1|)}, \\
R :=   \frac{\sinh^2(|\Theta_2|)(n-1)}{4 |\Theta _2|^2} \Bigl(|\Theta_2|^2
   \Bigl((n-1) \coth ^2(|\Theta_2| t/|\Theta _1|)-2 \csch
   ^2(|\Theta| _2 t/|\Theta| _1)\Bigr)\\ \hspace{7.5cm}-|\Theta _1|^2 \Bigl((n-1) \coth ^2(t) - 2 \csch ^2(t)\Bigr) \Bigr).\\
\end{gathered}
\end{equation*}
The bounds on these coefficient functions follow from the fact that all of $\sinh^2(x)/x^2$, $\sinh^2(x)/\sinh^2(xt/|\Theta_1|)$ and $y^2 \left((n-1) \coth ^2y -2 \csch
^2y)\right) $ are increasing functions on $(0, \infty)$, proven similarly as in the spherical case. The rest of the proof is identical, and we conclude that $\lambda_k(\Theta)\sinh^2(\Theta)$ decreases strictly on $(-\infty, 0)$.

\subsection*{Continuity and limits as $\Theta \to -\infty, 0, \pi$}
Continuity away from $\Theta = 0$ follows from the fact $\lambda_k(\Theta)$, $\sin^2(\Theta)$ and $\sinh^2(\Theta)$ are all continuous. Near $\Theta = 0$, $\sin \Theta$ is well-approximated by $\Theta$, and so the Rayleigh quotient for $\lambda_k(\Theta)\sin^2(\Theta)$ is uniformly approximated near $\Theta = 0$ by
\[
\frac{ \Theta^2 \displaystyle\int_{\S^{n-1}}\int_0^\Theta \left(|\partial_1 f|^2 + \frac{1}{\theta^2}|\grad_{\S^{n-1}} f|^2\right)\theta^{n-1} \, d\theta d\xi}{ \displaystyle \int_{\S^{n-1}}\int_0^\Theta |f|^2\theta^{n-1} \, d\theta d\xi}.
\]
If $\lambda_k(\B^n)$ is the eigenvalue of the Euclidean Laplacian for $\B^n$, then the above quantity is the Rayleigh quotient for $\Theta^2 \lambda_k(\Theta \B^n)$ in spherical coordinates. By Euclidean scale invariance, $\Theta^2 \lambda_k(\Theta \B^n) = \lambda_k(\B^n)$. Apply the corresponding argument to $\sinh |\Theta|$ and continuity at the origin follows.

When $\Theta > \pi/2$, then $0 \leq \lambda_k(\Theta)\sin^2(\Theta) \leq \lambda_k(\pi/2)\sin^2(\Theta)$ by Dirichlet domain monotonicity and hence $\lambda_k(\Theta)\sin^2(\Theta) \to 0$ as $\Theta \to \pi$

As $\Theta \to -\infty$, it is sufficient to show $\lambda_1(\Theta)\sinh^2(\Theta) \to \infty$. From Borisov and Freitas \cite[Theorem 3.3]{BF17}, we have $\lambda_1(\Theta) = \pi^2/\Theta^2 + 1$ for $n=3$ and the statement follows easily there. For $n \geq 4$, they also provide the lower bound
\[
\lambda_1(\Theta) \geq \frac{j^2_{\frac{n-2}{2}, 1}}{\Theta^2} + \frac{n(n-1)}{6}
\]
and again the result is clear.

 \section*{Acknowledgments}
The author was supported by National Science Foundation award {\#}2246537 to Richard Laugesen.

\appendix
\section{Coordinate representations}\label{coords}
\begin{center}
\begin{tabular}{l l}
   Spherical  &  Hyperbolic\\[2ex]
 $
\grad_{\S ^n} = \partial_\theta \hat{e}_\theta + \displaystyle\frac{1}{\sin \theta} \grad_{\S ^{n-1}}
$ & $\grad_{\H ^n} = \partial_\theta \hat{e}_\theta + \displaystyle\frac{1}{\sinh \theta} \grad_{\S ^{n-1}}$ \\ [3ex]
$\lapl_{\S^ n} = \displaystyle\frac{1}{\sin^{n-1}\theta}\partial_\theta(\sin^{n-1}\theta \,\partial_\theta) + \displaystyle\frac{1}{\sin^2\theta}\lapl_{\S ^{n-1}}$ & $\lapl_{\H^ n} = \displaystyle\frac{1}{\sinh^{n-1}\theta}\partial_\theta(\sinh^{n-1}\theta \,\partial_\theta) + \displaystyle\frac{1}{\sinh^2\theta}\lapl_{\S ^{n-1}}$ \\ [3ex]
$dV = \sin^{n-1}\theta \, d\theta d\xi$ & $dV = \sinh^{n-1}\theta \, d\theta d\xi$
\end{tabular}
\end{center}

\section{Numerics}
The plots in this section accomplish the following:

\begin{itemize}
    \item Illustrate Theorem \ref{main} in Figure \ref{dirichletlambda1graph}.
    
    \item Numerically support Conjectures \ref{neumannconj}, \ref{surfaceareaconjdirch}, and \ref{surfaceareaconjneum} in Figures \ref{neumannlambdagraph}, \ref{surfacemonofigs}, and \ref{neummanincreasingfig} respectively. 
    
    \item Highlight non-monotonicity of certain eigenvalue functionals in Figures \ref{geononmono} and \ref{stereononmono}. 
\end{itemize}
The plots were generated by using the \texttt{NDEigenvalues} command in Wolfram Mathematica 14.1.

\begin{figure}[htbp]
    \centering
    \includegraphics[width=0.49\linewidth]{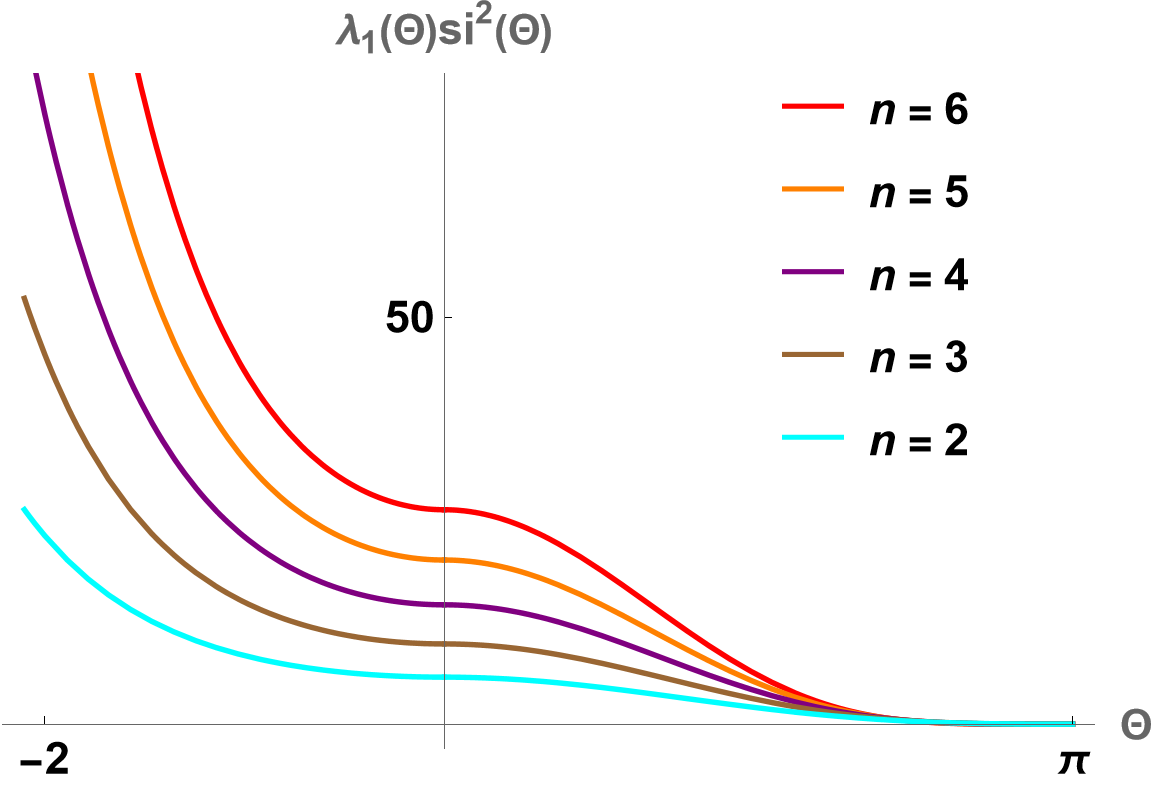}
    \includegraphics[width=0.49\linewidth]{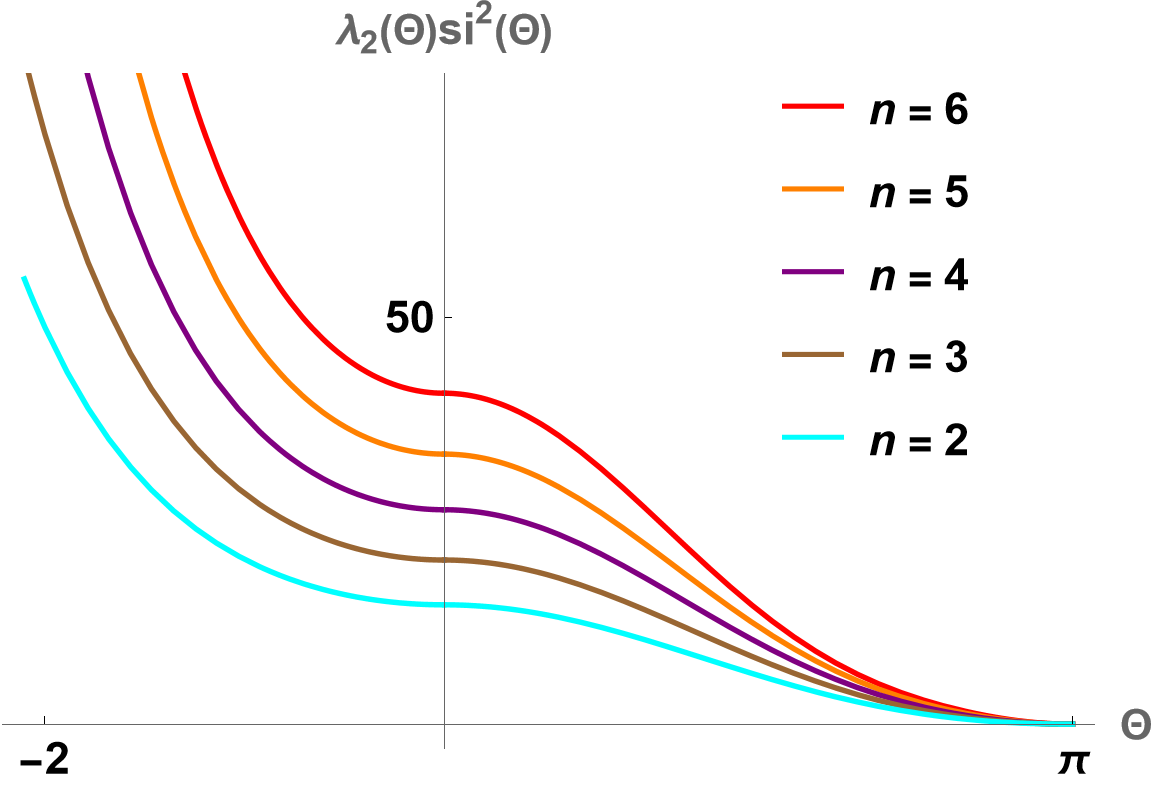}

    \caption{Graphs of the first and second Dirichlet eigenvalues as a function of the aperture (or geodesic radius) $\Theta$ in different dimensions, with $\si(\Theta)$ scaling. Monotonicity in dimension $n = 2$ was proven by Langford and Laugesen in \cite[Theorem 1(iii)]{LL22}. The higher dimensional case $n \geq 3$ is proven in this paper; see Theorem \ref{main}. Both results handle $\lambda_k$ for all $k$.} 
    \label{dirichletlambda1graph}
    
\end{figure}

\begin{figure}[]
    \centering
    
\includegraphics[width=0.9\linewidth]{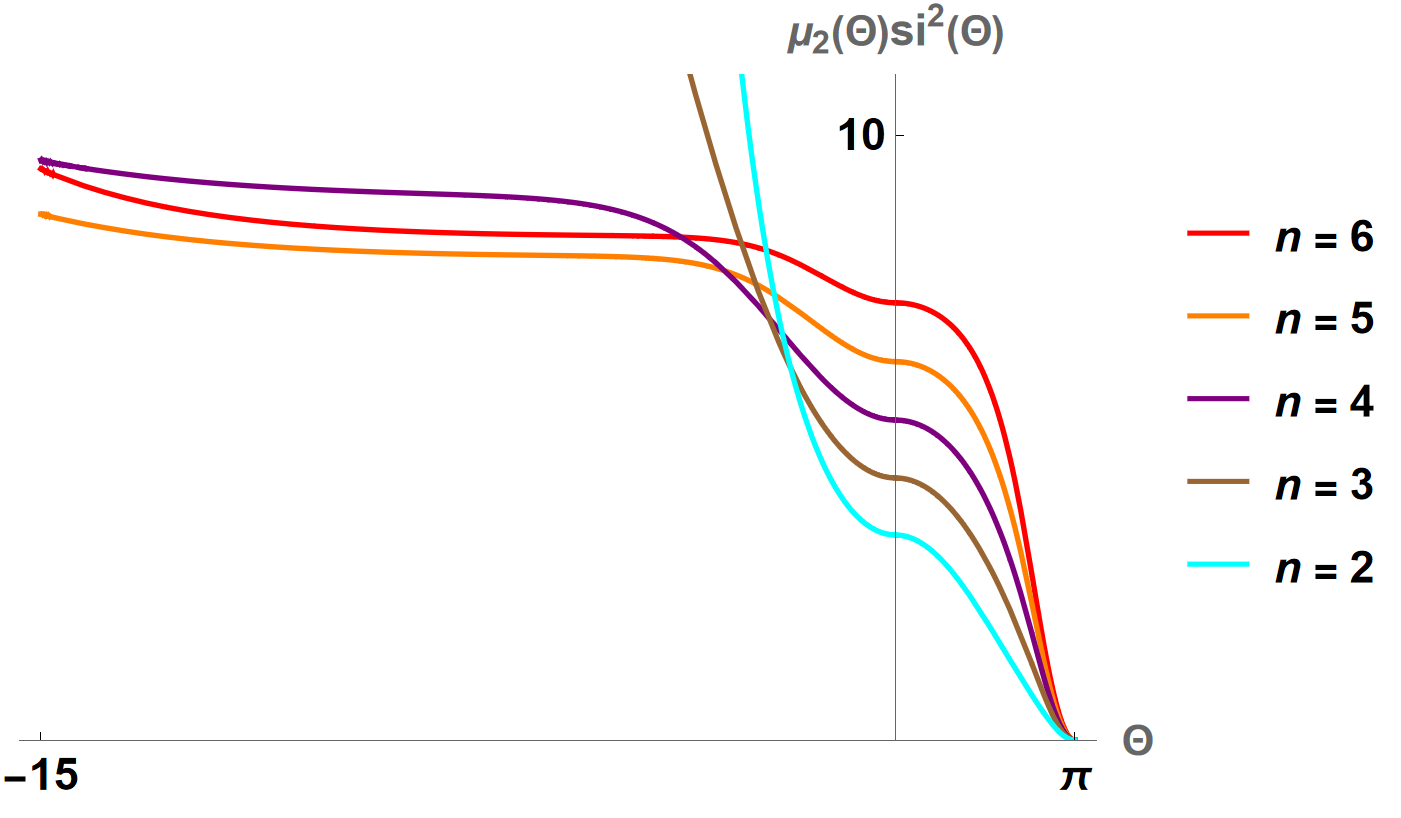}
    \caption{Graphs of the first non-zero Neumann eigenvalue as a function of $\Theta$ in different dimensions, with $\si^2(\Theta)$ scaling. There is some numerical instability for $\Theta < -15$. Monotonicity in dimension $n=2$ was proven by Langford and Laugesen \cite[Theorem 1(ii)]{LL22}. The higher dimensions are not proven; see Conjecture \ref{neumannconj}.}
    \label{neumannlambdagraph}
    
\end{figure}

\begin{figure}[]
\centering

\includegraphics[width=0.6\linewidth]{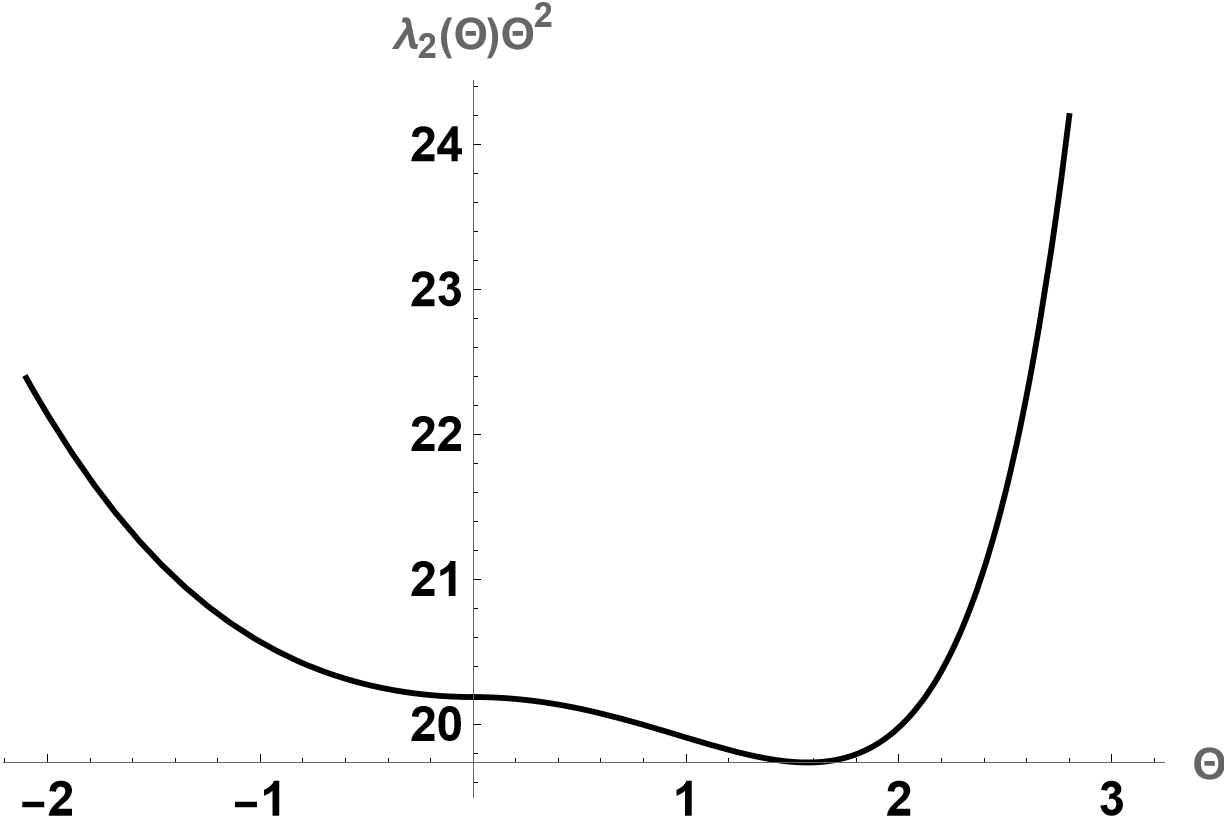}
    \caption{Graph of the second Dirichlet eigenvalue as a function of $\Theta$ in dimension $n = 3$. The scaling factor $\Theta^2$ is the square of geodesic radius. The functional here is non-monotonic. (For $n = 2$, see Langford and Laugesen \cite[Theorem 4]{LL22}.) See Table \ref{scaletab}.}
    \label{geononmono}
\end{figure}

\begin{figure}[]
\centering
\includegraphics[width=0.6\linewidth]{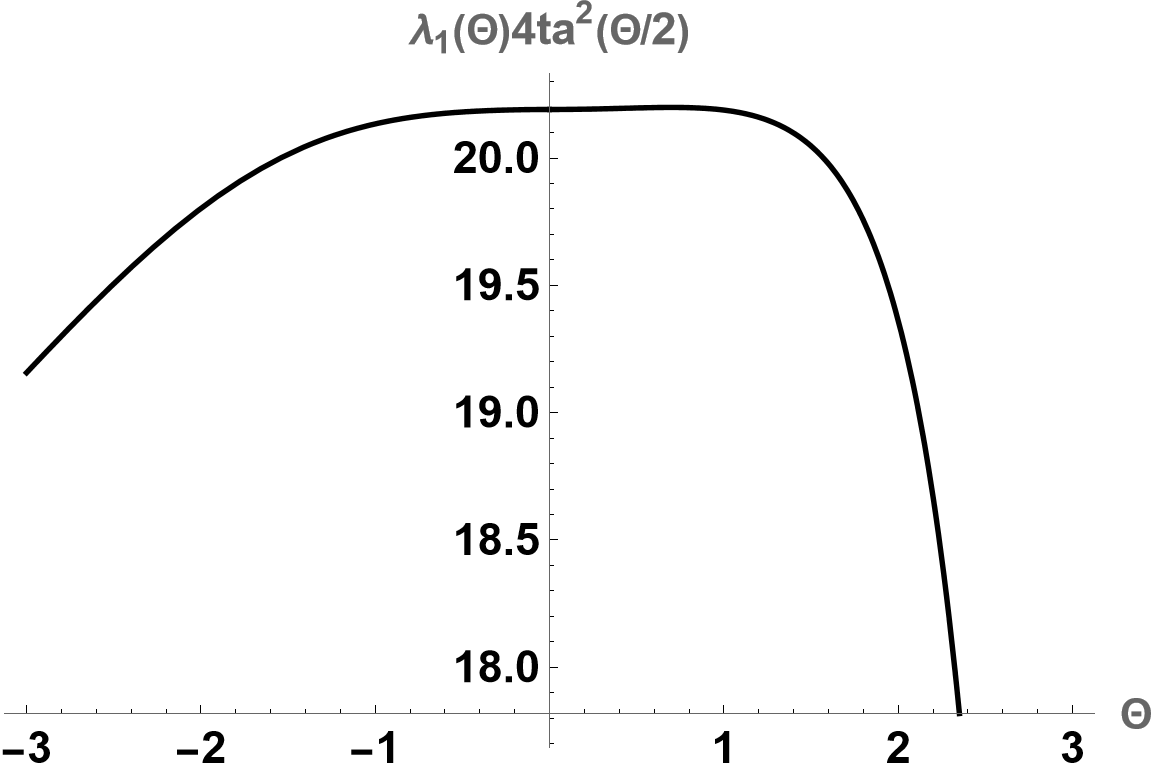}
    \caption{Graph of the first Dirichlet eigenvalue as a function of $\Theta$ in dimension $n = 5$. The scaling factor $4\ta^2(\Theta/2)$ is stereographic radius squared, and the functional here is non-monotonic. See Table \ref{scaletab}.}
    \label{stereononmono}
\end{figure}

\begin{figure}[]
    \centering
    \includegraphics[width=0.9\linewidth]{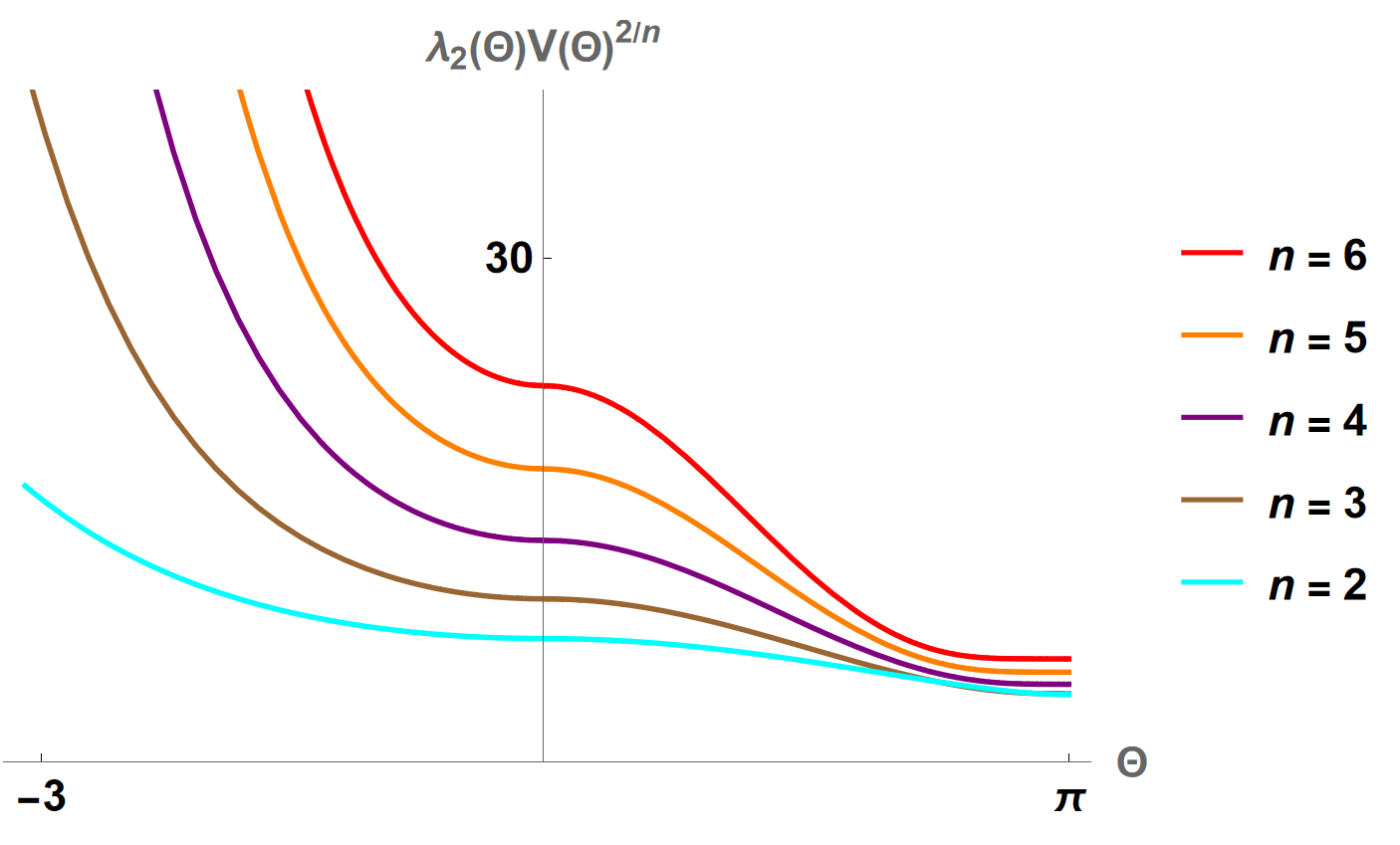}
    
    \caption{Graphs of the second Dirichlet eigenvalue as a function of $\Theta$ in multiple dimensions. The scaling factor is volume $V(\Theta) = c_n\int_0^{|\Theta|} \si(\theta)^{n-1}\, d \theta$.{} Monotonicity in dimension $n = 2$ was proven by Langford and Laugesen in \cite[Theorem 4]{LL22} for the interval $(-\infty, 0).$ The rest of the interval and higher dimensions are not proven; see Conjecture \ref{surfaceareaconjdirch}.}
    \label{surfacemonofigs} 
\end{figure}

\begin{figure}[]
    \centering
    \includegraphics[width=0.9\linewidth]{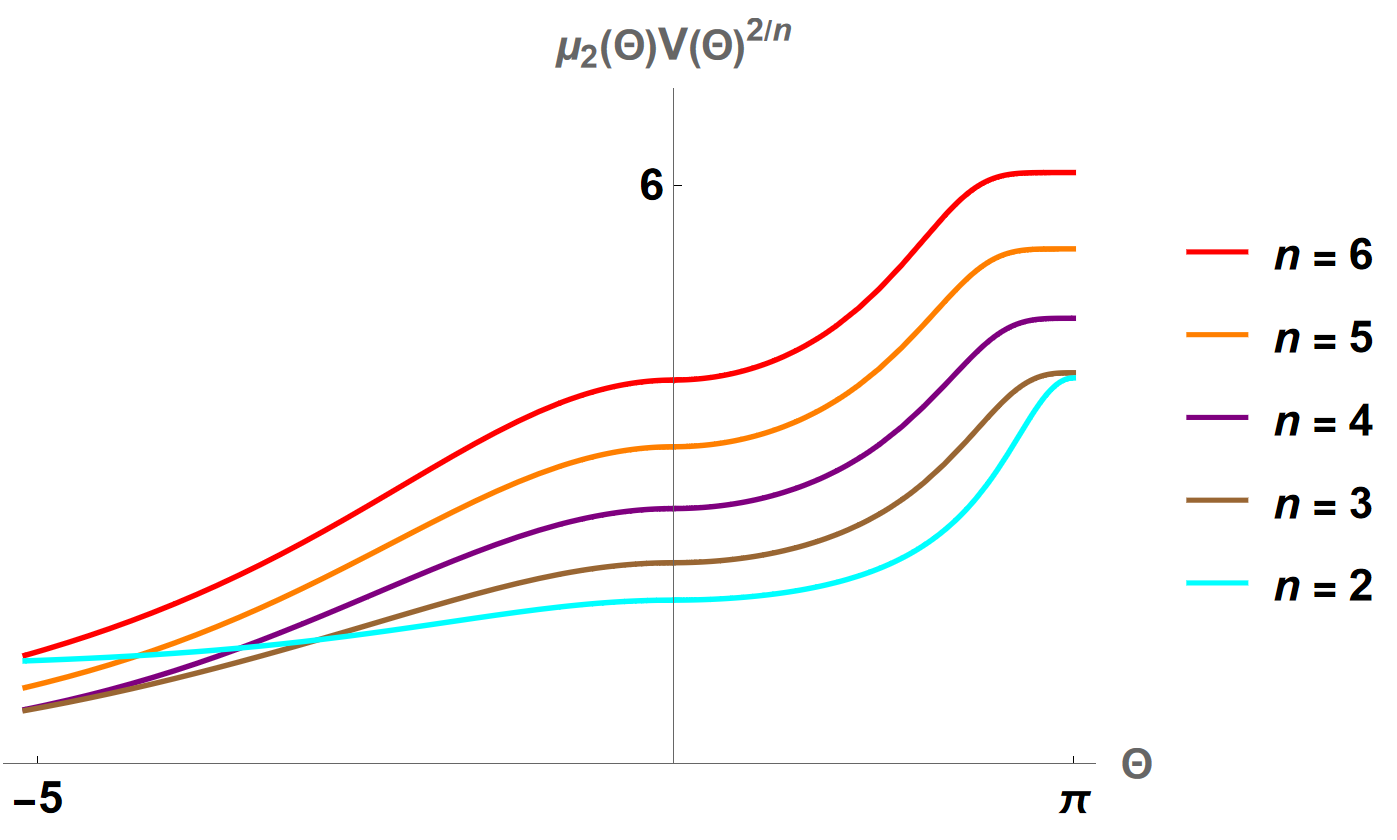}
    
    \caption{Graphs of the first non-zero Neumann eigenvalue as a function of $\Theta$ in multiple dimensions, with $V(\Theta)$ scaling. Monotonicity in dimension $n = 2$ was proven by Langford and Laugesen in \cite[Theorem 2]{LL22}. The higher dimensions are not proven; see Conjecture \ref{surfaceareaconjneum}.}
    \label{neummanincreasingfig}
\end{figure}

\end{document}